\documentclass{amsart}

\usepackage{graphicx}
\usepackage[shortlabels]{enumitem}

\usepackage{amsthm}
\usepackage{amsmath}
\usepackage{amssymb}

\usepackage{color}

\usepackage{tikz}
\usetikzlibrary{arrows,chains,matrix,positioning,scopes}
\usetikzlibrary{arrows.meta}
\usetikzlibrary{matrix}

\usepackage{hyperref}
\usepackage {cleveref}

\newtheorem{theorem}{Theorem}[section]
\newtheorem{proposition}[theorem]{Proposition}
\newtheorem{lemma}[theorem]{Lemma}
\newtheorem{corollary}[theorem]{Corollary}

\theoremstyle{definition}
\newtheorem{definition}[theorem]{Definition}

\theoremstyle{remark}
\newtheorem{remark}[theorem]{Remark}
\newtheorem{example}[theorem]{Example}

\begin{document}

\title[A Fundamental Class for Intersection Spaces]{A Fundamental Class for Intersection Spaces of Depth One Witt Spaces}

\author[Dominik J. Wrazidlo]{Dominik J. Wrazidlo}

\address{Institute of Mathematics for Industry, Kyushu University, Motooka 744, Nishi-ku, Fukuoka 819-0395, Japan}
\email{d-wrazidlo@imi.kyushu-u.ac.jp}



\begin{abstract}
By a theorem of Banagl-Chriestenson, intersection spaces of depth one pseudomanifolds exhibit generalized Poincar\'{e} duality of Betti numbers, provided that certain characteristic classes of the link bundles vanish.
In this paper, we show that the middle-perversity intersection space of a depth one Witt space can be completed to a rational Poincar\'{e} duality space by means of a single cell attachment, provided that a certain rational Hurewicz homomorphism associated to the link bundles is surjective.
Our approach continues previous work of Klimczak covering the case of isolated singularities with simply connected links.
For every singular stratum, we show that our condition on the rational Hurewicz homomorphism implies that the Banagl-Chriestenson characteristic classes of the link bundle vanish.
Moreover, using Sullivan minimal models, we show that the converse implication holds at least in the case that twice the dimension of the singular stratum is bounded by the dimension of the link.
As an application, we compare the signature of our rational Poincar\'{e} duality space to the Goresky-MacPherson intersection homology signature of the given Witt space.
We discuss our results for a class of Witt spaces having circles as their singular strata.
\end{abstract}

\maketitle
\tableofcontents

\section{Introduction}
The method of intersection spaces has been introduced by Banagl in \cite{ban2, ban} to provide a spatial perspective on Poincar\'{e} duality for singular spaces.
Following Banagl's original idea, such a theory should assign to a given singular space $X$ a family of \emph{intersection spaces} -- namely, spaces $I^{\overline{p}}X$ parametrized by a so-called perversity function $\overline{p}$ -- in such a way that, when $X$ is a closed, oriented $n$-dimensional pseudomanifold, generalized Poincar\'{e} duality $\widetilde{H}_{\ast}(I^{\overline{p}}X; \mathbb{Q}) \cong \widetilde{H}^{n-\ast}(I^{\overline{q}}X; \mathbb{Q})$ holds for the reduced singular (co)homology groups across complementary perversities $\overline{p}$ and $\overline{q}$.
Recall that this generalized form of Poincar\'{e} duality involving perversity functions originates from the well-established intersection homology theory $IH_{\ast}^{\overline{p}}(X; \mathbb{Q})$ of Goresky-MacPherson \cite{gm, gm2}.

The purpose of this paper is to upgrade the middle perversity intersection space of a depth one Witt space to a rational Poincar\'{e} duality space.
The fundamental class will be constructed by a single cell attachment in the top degree.
Before discussing our results (see \Cref{Outline of results}), which build on work of Klimczak \cite{klim} and Banagl-Chriestenson \cite{bc}, we give in the following an outline of several existing results in the theory of intersection spaces.

In comparison with Banagl's intersection space homology theory $HI_{\ast}^{\overline{p}}(X; \mathbb{Q}) = H_{\ast}(I^{\overline{p}}X; \mathbb{Q})$ one can observe that Goresky-MacPherson's intersection homology, and also Cheeger's $L^{2}$ cohomology of Riemannian pseudomanifolds \cite{che, che2, che3}, arise from certain intermediate algebraic chain complexes rather than from spatial modifications.
The intersection space construction itself modifies a space only near its singular strata:
Loosely speaking, each singularity link is replaced by a spatial approximation that truncates its homology in a degree dictated by the perversity function.
Motivation for using such spatial homology truncations (Moore approximations) of the links comes from a similar behavior of intersection homology groups in the case of isolated singularities (see Section 2 in \cite{ban2}).
As it turns out, the homology theories $HI_{\ast}^{\overline{p}}$ and $IH_{\ast}^{\overline{p}}$ (as well as their corresponding cohomology theories) are in general not isomorphic.
However, at least in the case of singular Calabi-Yau $3$-folds, they are related via mirror symmetry (see \cite{ban}).

Whenever intersection spaces exists, they serve as a source of desirable features that are not available in the context of intersection homology.
For instance, intersection space cohomology comes automatically equipped with a perversity internal cup product.
Moreover, addressing a problem suggested by Goresky and MacPherson in \cite{gm3}, intersection spaces provide an approach to construct generalized homology theories for singular spaces, like intersection $K$-theory (see Chapter 2.8 in \cite{ban}, as well as \cite{spie}).
Naturally, the advantages of the theory $HI_{\ast}^{\overline{p}}$ over $IH_{\ast}^{\overline{p}}$ come at the cost that the existence of intersection spaces which satisfy generalized Poincar\'{e} duality is far from granted.
For pseudomanifolds with isolated singularities, intersection spaces do always exist, and their duality theory is well-studied \cite{ban}.
However, for pseudomanifolds with more complicated singularities, the implementation of intersection space theory becomes rapidly more involved.
This is already evident in the case of arbitrary two strata pseudomanifolds:
Surprisingly, even if an intersection spaces can be constructed, the existence of a generalized Poincar\'{e} duality isomorphism turns out to be obstructed in general.
As discovered by Banagl and Chriestenson \cite{bc}, the failure of duality is precisely measured by \emph{local duality obstructions}, which are certain characteristic classes associated to the link bundle over the singular stratum of the pseudomanifold.
These obstruction classes are abstractly definable for fiberwise truncatable fiber bundles, and they vanish for product bundles and certain flat bundles, but not for generally twisted bundles.
For some specific three strata pseudomanifolds with bottom stratum a point, a duality result for intersection spaces has been established in \cite{ban3}.
By developing an inductive method of intersection space pairs, Agust\'{i}n and de Bobadilla have recently proposed in \cite{ab} a quite general construction of intersection spaces for pseudomanifolds of arbitrary stratification depth, at least when the link bundles can be compatibly trivialized.
However, an obstruction theory for generalized Poincar\'{e} duality of intersection space pairs is not known (compare problem (6) in Section 2.6 in \cite{ab}).

In view of the difficulties that arise in constructing intersection spaces, it seems beneficial to study intersection space homology by means of techniques that avoid constructing the intersection space itself.
Notable alternative approaches are via $L^{2}$ theory \cite{bh}, via linear algebra \cite{ges}, via sheaf theory \cite{ab}, and via differential forms \cite{ban4, ess}.
The approach via $L^{2}$ theory applies to two strata pseudomanifolds having trivial link bundle.
As for the linear algebra approach, Geske \cite{ges} constructs so-called algebraic intersection spaces on the chain level.
His construction is based on a generalization of Moore approximations to multiple degrees that might in general not be realizable as a spatial modification of a tubular neighborhood of the singular set.
While algebraic intersection spaces that satisfy generalized Poincar\'{e} duality exist for a large class of Whitney stratified pseudomanifolds, they are remote from the spatial concept in that they do not exhibit a natural cup product on cohomology, and the generalized local duality obstructions of Banagl-Chriestenson can be shown to vanish for an appropriate choice of the local intersection approximation (see Theorem 4.10 in \cite{ges}).
On the level of homology, algebraic intersection spaces turn out to be non-isomorphic to the intersection space pair approach of Agust\'{i}n and de Bobadilla (see Section 6 in \cite{ges}).
Note that in \cite{ab}, Agust\'{i}n and de Bobadilla pursue a sheaf theoretic approach that is inspired by work of Banagl, Maxim and Budur \cite{bm0, bm, bbm, max}.
Namely, they associate to intersection space pairs certain constructible sheaf complexes on the original pseudomanifold satisfying axioms analogous to those of the intersection chain complex in intersection homology theory \cite{gm2}.
Then, in Theorem 10.6 in \cite{ab} they show that so-called general intersection space complexes give rise to generalized Poincar\'{e} duality for two strata pseudomanifolds.
Finally, concerning the differential form approach, we note the special and important feature that wedge product of forms followed by integration induces a \emph{canonical} non-degenerate intersection pairing on cohomology in analogy with ordinary de Rham cohomology.

Returning to the original spatial approach, Klimczak \cite{klim} pursues the idea to realize Poincar\'{e} duality for intersection spaces by cup product followed by evaluation with a fundamental class rather than only showing equality of complementary Betti numbers.
Let us consider the important case of a Witt space $X$ with isolated singularities.
In this case, the intersection spaces associated to the lower middle and upper middle perversities $\overline{m}$ and $\overline{n}$ exist, and can be chosen to be equal, $IX = I^{\overline{m}}X = I^{\overline{n}}X$.
By a \emph{Klimczak completion} of $IX$ we shall mean a rational Poincar\'{e} duality space of the form $\widehat{IX} = IX \cup e^{n}$, where $n$ denotes the dimension of $X$.
If $IX$ admits a Klimczak completion, then the fundamental class in $H_{n}(\widehat{IX})$ arises from the newly attached top-dimensional cell $e^{n}$, and an easy Mayer-Vietoris argument implies that the Betti numbers of $IX$ and $\widehat{IX}$ agree in degrees $1, \dots, n-1$.
In \cite{klim} it is shown that Klimczak completions exist for compact Witt spaces having isolated singularities with simply connected links.
Moreover, the rational homotopy type of a Klimczak completion is determined by the intersection space whenever a theorem of Stasheff \cite{sta} is applicable.
In view of future applications it seems interesting to invoke rational surgery theory to realize Klimczak completions by manifolds.

In this paper we study Klimczak completions for middle perversity intersection spaces of compact depth one Witt spaces.
Future study will have to clarify the obstructions to the existence of Klimczak completions in the case of higher stratification depth.

The paper is structured as follows.
\Cref{Outline of results} presents our main results in case of a two strata Witt space.
In \Cref{general notation} we list some notation that will be used throughout the paper.
\Cref{truncation cones} and \Cref{rational homotopy theory} contain the proofs of our main technical results.
Finally, in \Cref{intersection spaces and the signature}, we prove our main results for depth one Witt spaces, and illustrate them in an example.

\textbf{Acknowledgements}
We are grateful to Markus Banagl and Timo Essig for several discussions.
Also, we thank Manuel Amann for a brief correspondence leading to \Cref{example minimal model}.

At the time this work was completed, the author was a JSPS International Research Fellow (Postdoctoral Fellowships for Research in Japan (Standard)).

\section{Statement of results}\label{Outline of results}
In this paper we study Klimczak completions for middle perversity intersection spaces of compact depth one Witt spaces.
For this purpose, we adopt the framework of Thom-Mather stratified spaces as presented in Section 8 of \cite{bc}.
For simplicity, we consider in the following a two strata Witt space $X$ with singular stratum $B$.
Then, the Thom-Mather control data induce a possibly twisted smooth fiber bundle $E \rightarrow B$ with fiber the link $L$ of $B$ such that the complement of a suitable tubular neighborhood of $B$ in $X$ is a smooth $n$-manifold $M$ with boundary $\partial M = E$.
In this setting, as explained in Section 10 of \cite{bc}, the intersection spaces $I^{\overline{m}}X$ and $I^{\overline{n}}X$ associated to the lower middle and upper middle perversities $\overline{m}$ and $\overline{n}$ exist and can be chosen to be equal, $IX = I^{\overline{m}}X = I^{\overline{n}}X$, provided that the fiber $L$ admits an \emph{equivariant} Moore approximation $f_{<} \colon L_{<} \rightarrow L$ of degree $\lfloor \frac{1}{2}(\operatorname{dim} L + 1)\rfloor$ (with respect to a suitable structure group for $E \rightarrow B$).
In view of Theorem 9.5 in \cite{bc} one might speculate that vanishing of Banagl-Chriestenson's local duality obstructions for $E \rightarrow B$ is in some way related to existence of a Klimczak completion for $IX$ because both assumptions imply Poincar\'{e} duality for the Betti numbers of $IX$.
In this context, a central role is played by the  \emph{truncation cone}, $\operatorname{cone}(F_{<})$, the mapping cone of the fiberwise truncation $F_{<} \colon \operatorname{ft}_{<}E \rightarrow E$ induced by $f_{<}$.
Namely, the local duality obstructions of the bundle $E \rightarrow B$ vanish if and only if all $(n-1)$-complementary cup products in $\widetilde{H}^{\ast}(\operatorname{cone}(F_{<}))$ vanish, where $n$ denotes the dimension of $X$.
Moreover, when $B$ is a point and $L$ is simply connected, then the construction of the Klimczak completion for the intersection space $IX = \operatorname{cone}(F_{<}) \cup_{\partial M} M$ is implemented in \cite{klim} (see Section 3.2.1 and also Proposition 3.11 therein) as follows.
In a first local step, an $n$-cell $e^{n}$ is attached to the truncation cone to produce a Poincar\'{e} duality pair $(\operatorname{cone}(F_{<}) \cup e^{n}, E)$, which is then in a second global step glued to the regular part $(M, \partial M)$ of $X$ to yield the desired rational Poincar\'{e} duality space $\widehat{IX}$.
Generalizing to an arbitrary singular stratum $B$ and arbitrary link $L$, our \Cref{main theorem} states that a Poincar\'{e} duality pair of the form $(\operatorname{cone}(F_{<}) \cup e^{n}, E)$ exists if and only if the rational Hurewicz homomorphism of the truncation cone in degree $n-1$,
\begin{align}\label{hurewicz truncation cone}
\operatorname{Hur}_{n-1 \ast} \colon \pi_{n-1}(\operatorname{cone}(F_{<}), \operatorname{pt}) \otimes_{\mathbb{Z}} \mathbb{Q} \rightarrow H_{n-1}(\operatorname{cone}(F_{<})),
\end{align}
is surjective.
Moreover, we show that the local duality obstructions for $E \rightarrow B$ vanish necessarily in that case, which reveals some part of their homotopy theoretic nature.
As a counterpart of Theorem 9.5 in \cite{bc}, our \Cref{main theorem} implies the following
\begin{theorem}\label{introduction main theorem}
Let $(X,B)$ be a compact two strata Witt space of dimension $n \geq 3$.
Assume that the link $L$ admits an equivariant Moore approximation $f_{<} \colon L_{<} \rightarrow L$ of degree $\lfloor \frac{1}{2}(\operatorname{dim} L + 1) \rfloor$.
If the rational Hurewicz homomorphism of the associated truncation cone (see (\ref{hurewicz truncation cone})) is surjective in degree $n-1$, then the middle perversity intersection space $IX$ admits a Klimczak completion.
\end{theorem}
More generally, our method applies to depth one Witt spaces with more than one singular stratum (see \Cref{main corollary}(a)).
Note that when $B$ is a point and $L$ is simply connected, we recover Klimczak's original result.
Recall that the argument in \cite{klim} uses the rational Hurewicz theorem to show that the rational Hurewicz homomorphism of the truncation cone is always surjective in the relevant degree.
We do not need to assume that the link is simply connected by employing the results of \cite{wra} for constructing Moore approximations for arbitrary path connected cell complexes.

According to \Cref{main theorem}, our surjectivity condition on the rational Hurewicz homomorphism is sufficient for the local duality obstructions to vanish.
However, even in the case of a globally trivial link bundle, we do in general not know whether the converse implication is also true.
Nevertheless, under the additional assumption that the truncation cone is simply connected, the  converse implication can be analyzed further by means of minimal Sullivan models from rational homotopy theory (see \Cref{corollary hurewicz theorem}).
In particular, in view of \Cref{main theorem}, an important consequence of our \Cref{proposition rational poincare duality space} is the following

\begin{theorem}[compare \Cref{main corollary}(b)]
Let $(X,B)$ be a compact two strata Witt space of dimension $n \geq 3$.
Assume that the link $L$ is path connected, and admits an equivariant Moore approximation $f_{<} \colon L_{<} \rightarrow L$ of degree $\lfloor \frac{1}{2}(\operatorname{dim} L+1) \rfloor$.
Suppose that $\operatorname{cone}(F_{<})$, the associated truncation cone, is simply connected.
If
\begin{align*}
n < 3 \cdot \operatorname{max}\{\lfloor  \frac{1}{2}(\operatorname{dim} L+1) \rfloor, \lceil \frac{1}{2}(\operatorname{dim} L+1) \rceil\} = \begin{cases}
\frac{3}{2}(\operatorname{dim} L+1), \quad &\operatorname{dim} L \text{ odd}, \\
\frac{3}{2}(\operatorname{dim} L+2), \quad &\operatorname{dim} L \text{ even},
\end{cases}
\end{align*}
or, equivalently,
\begin{align*}
\operatorname{dim}B = n-1-\operatorname{dim} L < \begin{cases}
\frac{1}{2}(\operatorname{dim} L+1), \quad &\operatorname{dim} L \text{ odd}, \\
\frac{1}{2}(\operatorname{dim} L+4), \quad &\operatorname{dim} L \text{ even},
\end{cases}
\end{align*}
then the following statements are equivalent:
\begin{enumerate}[(i)]
\item The local duality obstructions of the link bundle vanish, that is, all $(n-1)$-complementary cup products in $\widetilde{H}^{\ast}(\operatorname{cone}(F_{<}))$ vanish.
\item The rational Hurewicz homomorphism of the truncation cone (see (\ref{hurewicz truncation cone})) is surjective in degree $n-1$.
\end{enumerate}
\end{theorem}

The assumption that the truncation cone is simply connected is valid in many cases of practical interest -- for instance, whenever the link has abelian fundamental group as pointed out in \Cref{remark abelian fundamental group}, or when the link bundle is trivial, see \Cref{remark trivial link bundle}.

If the dimension of the Witt space $X$ is of the form $n = 4d$, then it is natural to study the symmetric intersection form $H_{2d}(\widehat{IX}) \times H_{2d}(\widehat{IX}) \rightarrow \mathbb{Q}$ of a Klimczak completion $\widehat{IX}$.
In accordance with the results of Section 11 in \cite{bc}, we can compare it to the Goresky-MacPherson-Siegel intersection form $IH_{2d}(X) \times IH_{2d}(X) \rightarrow \mathbb{Q}$ on middle-perversity intersection homology (see Section I.4.1 in \cite{sieg}) as follows.

\begin{theorem}[compare \Cref{main corollary}(c)]
Let $(X,B)$ be a compact two strata Witt space of dimension $n = 4d$.
Suppose that the rational Hurewicz homomorphism of the truncation cone (see (\ref{hurewicz truncation cone})) is surjective in degree $n-1$.
Then, the Witt element $w_{HI} \in W(\mathbb{Q})$ induced by the symmetric intersection form $H_{2d}(\widehat{IX}) \times H_{2d}(\widehat{IX}) \rightarrow \mathbb{Q}$ of the Poincar\'{e} duality space $\widehat{IX}$ equals the Witt element $w_{IH} \in W(\mathbb{Q})$ induced by the Goresky-MacPherson-Siegel intersection form $IH_{2d}(X) \times IH_{2d}(X) \rightarrow \mathbb{Q}$ on middle-perversity intersection homology.
In particular, it follows that the two intersection forms have equal signatures.
\end{theorem}

We point out that our proof of \Cref{main corollary}(c) exploits massively the existence of a fundamental class for $\widehat{IX}$, which enables us to invoke Novikov additivity for Poincar\'{e} duality pairs (see Lemma 3.4 in \cite{klim}).
On the other hand, lacking the existence of a fundamental class under the assumption that the local duality obstructions vanish, the argument of Banagl-Chriestenson in Section 11 in \cite{bc} requires an involved construction of an abstract, non-canonical symmetric intersection form for $IX$.
It seems to be an interesting problem to compare intersection forms of Klimczak completions to intersection forms that arise from the differential form approach \cite{ban4}.

In \Cref{An example} we provide a class of examples of depth one Witt spaces with twisted link bundles for which our \Cref{main corollary} applies.
For further examples concerning the existence of equivariant Moore approximations in general we refer to Sections 3 and 12 in \cite{bc}.

\section{General notation}\label{general notation}
We collect some general notation that will be used throughout the paper.

By a pair of spaces $(X, A)$ we mean a topological space $X$ together with a subspace $A \subset X$.
A pointed pair of spaces $(X, A, x_{0})$ is a pair of spaces $(X, A)$ together with a basepoint $x_{0} \in A$.
A map of pairs $f \colon (X, A) \rightarrow (X', A')$ is a map $f \colon X \rightarrow X'$ such that $f(A) \subset A'$.

Let $D^{p} = \{x \in \mathbb{R}^{p}; x_{1}^{2} + \dots + x_{p}^{2} \leq 1\}$ denote the closed unit ball in Euclidean $p$-space $\mathbb{R}^{p}$, and $S^{p-1} := \partial D^{p}$ the standard $(p-1)$-sphere.
We also fix $s_{0} = 1 \in S^{0} \subset S^{p-1}$ as a basepoint.

Given a pointed pair of spaces $\left(X, A, x_{0}\right)$, the Hurewicz map in degree $n \geq 1$ is
\begin{displaymath}
\operatorname{Hur}_{n} \colon \pi_{n}\left(X, A, x_{0}\right) \rightarrow H_{n}\left(X, A; \mathbb{Z}\right), \quad \operatorname{Hur}_{n}\left(\left[f\right]\right) = f_{\ast}\left(\nu\right),
\end{displaymath}
where $f_{\ast} \colon H_{n}\left(D^{n}, \partial D^{n}; \mathbb{Z}\right) \rightarrow H_{n}\left(X, A; \mathbb{Z}\right)$ is induced by a representative
$$
f \colon \left(D^{n}, \partial D^{n}, s_{0}\right) \rightarrow \left(X, A, x_{0}\right)
$$
of $[f] \in \pi_{n}\left(X, A, x_{0}\right)$, and $\nu$ denotes a fixed generator of $H_{n}\left(D^{n}, \partial D^{n}; \mathbb{Z}\right) \cong \mathbb{Z}$.

For a pair of spaces $(X, A)$, we will denote by $H_{i}(X, A)$ and $H^{i}(X, A)$ the $i$-th homology and cohomology groups with rational coefficients, respectively.
Using the canonical identifications $H_{i}(X, A) = H_{i}(X, A; \mathbb{Z}) \otimes_{\mathbb{Z}} \mathbb{Q}$ and $H^{i}(X, A) = \operatorname{Hom}_{\mathbb{Z}}(H_{i}(X, A; \mathbb{Z}), \mathbb{Q})$, we will also write
\begin{displaymath}
\operatorname{Hur}_{n \ast} = \operatorname{Hur}_{n} \otimes_{\mathbb{Z}} \mathbb{Q} \colon \pi_{n}\left(X, A, x_{0}\right) \otimes_{\mathbb{Z}} \mathbb{Q} \rightarrow H_{n}\left(X, A\right)
\end{displaymath}
and
\begin{displaymath}
\operatorname{Hur}_{n}^{\ast} = \operatorname{Hom}_{\mathbb{Z}}(\operatorname{Hur}_{n}, \mathbb{Q}) \colon H^{n}\left(X, A\right) \rightarrow \operatorname{Hom}_{\mathbb{Z}}(\pi_{n}\left(X, A, x_{0}\right), \mathbb{Q}).
\end{displaymath}

Given maps  $X \xleftarrow{f} A \xrightarrow{g} Y$ between topological spaces, we define the \emph{homotopy pushout} of $f$ and $g$ (see Section 2 in \cite{bc}) to be the topological space $X \cup_{A}Y$ defined as the quotient of the disjoint union $A \times [0, 1] \sqcup X \sqcup Y$ by the smallest equivalence relation generated by $\{(a, 0) \sim f(a)| \; a \in A\} \cup \{(a, 1) \sim g(a)| \; a \in A\}$.
In particular, if $X = \operatorname{pt}$ is the one-point space, then $\operatorname{pt} \cup_{A}Y = \operatorname{cone}(g \colon A \rightarrow Y) = \operatorname{cone}(g)$ is just the mapping cone (the homotopy cofiber) of $A \xrightarrow{g} Y$.
If, in addition, $Y = A$ and $g = \operatorname{id}_{A}$, then $\operatorname{pt} \cup_{A}A = \operatorname{cone}(\operatorname{id}_{A}) = \operatorname{cone}(A)$ is just the cone of the space $A$.
The inclusion $A \times \{1\} \subset A \times [0, 1]$ induces a canonical inclusion $A \subset \operatorname{cone}(A)$.
Moreover, given a map $g \colon A \rightarrow Y$, the homotopy pushout of $A \xleftarrow{\operatorname{id}_{A}} A \xrightarrow{g} Y$ is just the mapping cylinder $\operatorname{cyl}(g)$ of $g$, and the inclusion $A \times \{0\} \subset A \times [0, 1]$ induces a natural inclusion $A \subset \operatorname{cyl}(g)$.

Let $[M] \in H_{n}(M, \mathbb{Z})$ denote the fundamental class of a closed oriented $n$-manifold $M^{n}$.
The image of $[M] \otimes 1$ under the canonical identification $H_{n}(M, \mathbb{Z}) \otimes \mathbb{Q} \cong H_{n}(M)$ will be denoted by $[M]$ as well.

\section{Trucation cones}\label{truncation cones}

Before stating \Cref{main theorem}, the main result of this section (see \Cref{proof of main theorem} for the proof), we explain the necessary notation taken from \cite{bc}.
Throughout this section, let $\pi \colon E \rightarrow B$ be a (locally trivial) fiber bundle of closed manifolds with closed manifold fiber $L$ and structure group $G$ such that $B$, $E$ and $L$ are compatibly oriented.
In our applications in \Cref{intersection spaces and the signature}, $\pi$ will arise as a link bundle of a depth one pseudomanifold $X$, where we utilize the setting of Thom-Mather stratified pseudomanifolds that is considered in the work of Banagl and Chriestenson (see Section 8 in \cite{bc}).

Recall that a perversity is a function $\overline{p} \colon \{2, 3, \dots\} \rightarrow \{0, 1, \dots\}$ which satisfies the Goresky-MacPherson growth conditions $\overline{p}(2) = 0$ and $\overline{p}(s) \leq \overline{p}(s + 1) \leq \overline{p}(s) + 1$ for all $s \in \{2,3, \dots\}$.
We fix two perversities $\overline{p}$ and $\overline{q}$, and require them to be complementary in the sense that $\overline{p}(s) + \overline{q}(s) = s-2 = \overline{t}(s)$ for all $s \in \{2, 3, \dots\}$, where $\overline{t}$ is called the top perversity.
Let $n-1 = \operatorname{dim} E$ and $c = \operatorname{dim} L$ denote the dimensions of the total space and the fiber, respectively.
We define the cut-off degrees $k = c-\overline{p}(c+1)$ and $l = c-\overline{q}(c+1)$, and note that $k + l = 2c - \overline{t}(c+1) = c+1$.

Recall from Definition 3.2 in \cite{bc} that a $G$-equivariant Moore approximation to $L$ of degree $r$ is a $G$-space $L_{<r}$ together with a $G$-equivariant map $L_{<r} \rightarrow L$ that induces isomorphisms $H_{i}(L_{<r}) \cong H_{i}(L)$ in degrees $i < r$, and such that $H_{i}(L_{<r}) = 0$ in degrees $i \geq r$.
We assume that $H_{i}(L) = 0$ for $i = \operatorname{min}\{k, l\}, \dots,\operatorname{max}\{k, l\}-1$, and that the fiber $L$ possesses a $G$-equivariant Moore approximation of degree $k$.
Equivalently, there is a map $f_{<} \colon L_{<} \rightarrow L$ which is a $G$-equivariant Moore approximation to $L$ both of degree $k$ and of degree $l$.
This situation is of interest in the important case that $\pi$ is the link bundle of a two strata Witt space, and $\overline{p} = \overline{m}$ and $\overline{q} = \overline{n}$ are the lower middle and upper middle perversities defined by $\overline{m}(s) = \lfloor s/2 \rfloor - 1$ and $\overline{n}(s) = \lceil s/2 \rceil - 1$ for all $s \in \{2, 3, \dots\}$, respectively (see Section 10 of \cite{bc}).
In this case it follows that $k = l = (c+1)/2$ for $c$ odd, and $\operatorname{min}\{k, l\} = c/2$ and $\operatorname{max}\{k, l\} = c/2 + 1$ for $c$ even.

Following the discussion leading to Definition 6.1 in \cite{bc}, we can consider the induced \emph{fiberwise truncation} (both of degree $k$ and of degree $l$)
$$
F_{<} \colon \operatorname{ft}_{<} E \rightarrow E.
$$
Recall that $\operatorname{ft}_{<}E$ is the total space of the fiber bundle $\pi_{<} \colon \operatorname{ft}_{<} E \rightarrow B$ obtained by replacing the fiber $L$ of $\pi$ with the fiber $L_{<}$ (by means of the $G$-action), and $F_{<} \colon \operatorname{ft}_{<} E \rightarrow E$ is induced by $f_{<} \colon L_{<} \rightarrow L$ (using $G$-equivariance) in such a way that $\pi \circ F_{<} = \pi_{<}$.
The mapping cone of the fiberwise truncation $F_{<}$ plays a central role in the theory.
In fact, according to Definition 9.1 in \cite{bc} the perversity $\overline{p}$ and perversity $\overline{q}$ intersection spaces of a two strata pseudomanifold $X^{n}$ are given by
$$
IX = I^{\overline{p}}X = I^{\overline{q}}X = \operatorname{cone}(F_{<}) \cup_{E} M,
$$
where $M$ is the $n$-dimensional manifold with boundary $\partial M = E$ that arises as the complement of a suitable tubular neighborhood of $B$ in $X$.
(Note that $IX$ is actually defined as the mapping cone of the composition of $F_{<}$ with the inclusion $E \subset M$, but this space can be seen to be homeomorphic to $\operatorname{cone}(F_{<}) \cup_{E} M$.)
Furthermore, we note that in our setting, the general definition of the local duality obstructions $\mathcal{O}_{\ast}(\pi, k, l)$ (see Definition 6.8 in \cite{bc}) reduces in degree $i$ to
$$
\mathcal{O}_{i}(\pi, k, l) = \{x \cup y | \; x \in \widetilde{H}^{i}(\operatorname{cone}(F_{<})), y \in \widetilde{H}^{n-1-i}(\operatorname{cone}(F_{<}))\},
$$
which is a linear subspace of $\widetilde{H}^{n-1}(\operatorname{cone}(F_{<}))$.
In particular, note that we have $\mathcal{O}_{i}(\pi, k, l) = 0$ if and only if $x \cup y = 0$ in $\widetilde{H}^{n-1}(\operatorname{cone}(F_{<}))$ for all $x \in \widetilde{H}^{i}(\operatorname{cone}(F_{<}))$ and $y \in \widetilde{H}^{n-1-i}(\operatorname{cone}(F_{<}))$.
Since the truncation cone $\operatorname{cone}(F_{<})$ is path connected (compare \Cref{lemma simply connected cone}), we have $\mathcal{O}_{i}(\pi, k, l) = 0$ for all $i \notin \{1, \dots, n-2\}$.

\begin{theorem}\label{main theorem}
Let $n \geq 3$ be an integer.
Let $\pi \colon E^{n-1} \rightarrow B$ be a fiber bundle of closed manifolds with closed manifold fiber $L$ and structure group $G$ such that $B$, $E$ and $L$ are compatibly oriented.
Suppose that $B$ admits a good open cover (this holds whenever $B$ is smooth or at least PL).
Let $\overline{p}$ and $\overline{q}$ be complementary perversities, and set $k = c-\overline{p}(c+1)$ and $l = c-\overline{q}(c+1)$, where $c = \operatorname{dim} L$.
Suppose that the fiber $L$ possesses a $G$-equivariant Moore approximation $f_{<} \colon L_{<} \rightarrow L$ which is both of degree $k$ and of degree $l$.
Let $F_{<} \colon \operatorname{ft}_{<} E \rightarrow E$ denote the induced fiberwise truncation of $\pi$.
Then, the following statements are equivalent:
\begin{enumerate}[(i)]
\item There exists an attaching map
$$
\phi \colon (S^{n-1}, s_{0}) \rightarrow (\operatorname{cone}(F_{<}), \operatorname{pt})
$$
and a lift $[e_{\phi}] \in H_{n}(\operatorname{cone}(F_{<}) \cup_{\phi} D^{n}, E)$ of the fundamental class $[E] \in H_{n-1}(E)$ such that $(\operatorname{cone}(F_{<}) \cup_{\phi} D^{n}, E)$ is a rational Poincar\'{e} duality pair of dimension $n$ with orientation class $[e_{\phi}]$ (see \Cref{definition poincare duality pair}).
\item The image of the fundamental class $[E] \in H_{n-1}(E)$ under the map
$$
H_{n-1}(E) \rightarrow H_{n-1}(\operatorname{cone}(F_{<}))
$$
induced by the inclusion $E \subset \operatorname{cone}(F_{<})$ is contained in the image of the rational Hurewicz homomorphism
$$
\operatorname{Hur}_{n-1 \ast} \colon \pi_{n-1}(\operatorname{cone}(F_{<}), \operatorname{pt}) \otimes_{\mathbb{Z}} \mathbb{Q} \rightarrow H_{n-1}(\operatorname{cone}(F_{<})).
$$
\end{enumerate}
Furthermore, if either of the above statements holds, then $\mathcal{O}_{i}(\pi, k, l) = 0$ for all $i$.
\end{theorem}

\begin{remark}\label{remark independence of attaching map}
If statement $(i)$ of \Cref{main theorem} holds and $E$ is connected, then the duality isomorphism
$$
- \cap [e_{\phi}] \colon H^{r}(\operatorname{cone}(F_{<}) \cup_{\phi} D^{n}) \stackrel{\cong}{\longrightarrow} H_{n-r}(\operatorname{cone}(F_{<}) \cup_{\phi} D^{n}, E)
$$
does not depend on the choice of the attaching map $\phi$ in the following sense.
In degrees $r = 0$ and $r = n$ the isomorphism $- \cap [e_{\phi}]$ is canonical because $1 \in H^{0}(\operatorname{cone}(F_{<}) \cup_{\phi} D^{n})$ is a generator, and $H_{0}(\operatorname{cone}(F_{<}) \cup_{\phi} D^{n}, E) = 0$.
Moreover, in degree $r = n-1$ the isomorphism $- \cap [e_{\phi}]$ is canonical because $H^{n-1}(\operatorname{cone}(F_{<}) \cup_{\phi} D^{n}) \cong \widetilde{H}_{0}(E) = 0$ by \Cref{remark homology calculations} and \Cref{remark connecting homomorphisms} applied to the map $f \colon Y \rightarrow Z$ given by $F_{<} \colon \operatorname{ft}_{<} E \rightarrow E$, and to its mapping cone $X = \operatorname{cone}(f) = \operatorname{cone}(F_{<})$.
(This is the only point where we use that $E$ is connected.)
Finally, in degrees $r \in \{1, \dots, n-2\}$ it follows from the last part of the proof of \Cref{main theorem} (see \Cref{proof of main theorem}) that there exist isomorphisms
$$
D_{r} \colon H^{r}(\operatorname{cone}(F_{<})) \stackrel{\cong}{\longrightarrow} H_{n-r}(\operatorname{cone}(F_{<}), E)
$$
such that for every attaching map $\phi$ there is the following commutative diagram:
\begin{center}
\begin{tikzpicture}

\path (0,0) node(a) {$H^{r}(\operatorname{cone}(F_{<}))$}
         (6, 0) node(b) {$H_{n-r}(\operatorname{cone}(F_{<}), E)$}
         (0, -2) node(c) {$H^{r}(\operatorname{cone}(F_{<}) \cup_{\phi} D^{n})$}
         (6, -2) node(d) {$H_{n-r}(\operatorname{cone}(F_{<}) \cup_{\phi} D^{n}, E)$.};

\draw[->] (c) -- node[left] {$\operatorname{incl}^{\ast} \; \cong$} (a);
\draw[->] (a) -- node[above] {$D_{r} \quad \cong$} (b);
\draw[->] (b) -- node[right] {$\cong \; \operatorname{incl}_{\ast}$} (d);
\draw[->] (c) -- node[above] {$- \cap [e_{\phi}]$} (d);
\end{tikzpicture}
\end{center}
\end{remark}

\begin{remark}[CW structures]\label{remark cw structures}
In applications it might be necessary to know that the Poincar\'{e} duality pair provided by statement (i) in \Cref{main theorem} is a CW pair.
To achieve this, we assume that the spaces $B$, $L$, and $L_{<}$ carry CW structures, and that the map $f_{<} \colon L_{<} \rightarrow L$ is cellular.
Then, note that $E$ and $\operatorname{ft}_{<} E$ inherit natural CW structures in such a way that the induced fiberwise truncation $F_{<} \colon \operatorname{ft}_{<} E \rightarrow E$ is cellular.
Thus, the truncation cone $\operatorname{cone}(F_{<})$ has a cell structure.
Finally, by applying the cellular approximation theorem to the map $\phi$ in \Cref{lemma hurewicz map}, we may without loss of generality assume that the map $\phi$ in statement (i) of \Cref{main theorem} is cellular.
Hence, it follows that $(\operatorname{cone}(F_{<}) \cup_{\phi} D^{n}, E)$ is a CW pair.
\end{remark}

\subsection{The rational Hurewicz homomorphism}
The following lemma is a slight modification of Lemma 3.8 in \cite[p. 248]{klim}, and will be used in the proofs of \Cref{main theorem} (see \Cref{proof of main theorem}) and \Cref{main corollary}.

\begin{lemma}\label{lemma hurewicz map}
Let $n \geq 3$ be an integer.
Let $\phi \colon (S^{n-1}, s_{0}) \rightarrow (X, x_{0})$ be a map of pointed spaces, and let $X^{\phi} = \operatorname{cone}(\phi)$ denote the associated mapping cone.
Then for any class $x \in H_{n-1}(X)$ the following statements are equivalent:
\begin{enumerate}[$(i)$]
\item There exists $q \in \mathbb{Q}$ such that the element $[\phi] \otimes q \in \pi_{n-1}(X, x_{0}) \otimes_{\mathbb{Z}} \mathbb{Q}$ is mapped to the class $x \in H_{n-1}(X)$ under the Hurewicz homomorphism
$$
\operatorname{Hur}_{n-1 \ast} \colon \pi_{n-1}(X, x_{0}) \otimes_{\mathbb{Z}}  \mathbb{Q} \rightarrow H_{n-1}(X).
$$
\item The class $x \in H_{n-1}(X)$ lies in the image of the connecting homomorphism
$$
\partial_{n} \colon H_{n}(X^{\phi}, X) \rightarrow H_{n-1}(X).
$$
\end{enumerate}
\end{lemma}

\begin{proof}
Consider the commutative diagram

\begin{center}
\begin{tikzpicture}

\path (0,0) node(a) {$\pi_{n}\left(X^{\phi}, X, x_{0}\right) \otimes_{\mathbb{Z}}  \mathbb{Q}$}
         (4, 0) node(b) {$\pi_{n-1}\left(X, x_{0}\right) \otimes_{\mathbb{Z}}  \mathbb{Q}$}
         (0, -2) node(c) {$H_{n}\left(X^{\phi}, X\right)$}
         (4, -2) node(d) {$H_{n-1}\left(X\right)$.};

\draw[->] (a) -- node[left] {$\operatorname{Hur}_{n \ast}$} (c);
\draw[->] (a) -- node[above] {$\partial_{n}$} (b);
\draw[->] (b) -- node[right] {$\operatorname{Hur}_{n-1 \ast}$} (d);
\draw[->] (c) -- node[above] {$\partial_{n}$} (d);
\end{tikzpicture}
\end{center}

By construction of $X^{\phi}$ we have a map of pointed pairs
$$
(\Phi, \phi) \colon (D^{n}, S^{n-1}, s_{0}) \rightarrow (X^{\phi}, X, x_{0}).
$$
Consider the induced element $[\Phi] \otimes 1 \in \pi_{n}\left(X^{\phi}, X, x_{0}\right) \otimes_{\mathbb{Z}}  \mathbb{Q}$.
The homomorphism
$$
\partial_{n} \colon \pi_{n}\left(X^{\phi}, X, x_{0}\right) \otimes_{\mathbb{Z}}  \mathbb{Q} \rightarrow \pi_{n-1}\left(X, x_{0}\right) \otimes_{\mathbb{Z}}  \mathbb{Q}
$$
maps $[\Phi] \otimes 1$ to the element $[\phi] \otimes 1 \in \pi_{n-1}\left(X, x_{0}\right) \otimes_{\mathbb{Z}}  \mathbb{Q}$.
The homomorphism
$$
\operatorname{Hur}_{n \ast} \colon \pi_{n}\left(X^{\phi}, X, x_{0}\right) \otimes_{\mathbb{Z}}  \mathbb{Q} \rightarrow H_{n}\left(X^{\phi}, X\right)
$$
maps $[\Phi] \otimes 1$ to a generator of $H_{n}\left(X^{\phi}, X\right) \cong \tilde{H}_{n}(X^{\phi}/X) \cong \tilde{H}_{n}(S^{n}) \cong \mathbb{Q}$.
(Here, the isomorphism $H_{n}(X^{\phi}, X) \cong \widetilde{H}_{n}(X^{\phi}/ X)$ holds by Proposition 2.22 in \cite[p. 124]{hat}.)
Thus, if statement $(i)$ holds, then we have
$$
x = \operatorname{Hur}_{n-1 \ast}([\phi] \otimes q) = \operatorname{Hur}_{n-1 \ast}(\partial_{n}([\Phi] \otimes q)) = \partial_{n}(\operatorname{Hur}_{n \ast}([\Phi] \otimes q)),
$$
and statement $(ii)$ follows.
Conversely, if statement $(ii)$ holds, then there exists $q \in \mathbb{Q}$ such that
$$
x = \partial_{n}(q \cdot \operatorname{Hur}_{n \ast}([\Phi] \otimes 1)) = \operatorname{Hur}_{n-1 \ast}(\partial_{n}([\Phi] \otimes q)) = \operatorname{Hur}_{n-1 \ast}([\phi] \otimes q),
$$
and statement $(i)$ follows.
\end{proof}

\begin{remark}
By passing to the abelianization $\pi_{1}(X, x_{0}) \rightarrow \pi_{1}(X, x_{0})_{ab}$, we can state \Cref{lemma hurewicz map} as well for $n = 2$.
However, the integer $n$ will arise in our applications in \Cref{intersection spaces and the signature} as the dimension of depth one pseudomanifolds.
As their singular strata are required to have codimension at least $2$ (see \Cref{Witt spaces and signature}) and the case of point strata is covered by \cite{klim}, we will generally assume that $n \geq 3$ throughout the paper.
\end{remark}

\subsection{Rational Poincar\'{e} duality pairs}
We recall the fundamental concept of Poincar\'{e} duality pairs of spaces (compare Section 3.1 in \cite{klim} and Section II.2 in \cite{bro}).
In this paper we do not require cell structures on our spaces, but see \Cref{remark cw structures}.

\begin{definition}\label{definition poincare duality pair}
A pair $(A, B)$ is called \emph{(rational) Poincar\'{e} duality pair of dimension $n$} if
\begin{enumerate}[(i)]
\item all homology groups $H_{r}(A)$ and $H_{r}(B)$, $r \in \mathbb{Z}$, are of finite rank, and
\item there exists a class $a \in H_{n}(A, B)$ such that
$$
- \cap a \colon H^{r}(A) \rightarrow H_{n-r}(A, B)
$$
is an isomorphism for all $r \in \mathbb{Z}$.
Any such class $a \in H_{n}(A, B)$ will be called an \emph{orientation class} for $(A, B)$.
\end{enumerate}
\end{definition}

\begin{remark}\label{remark poincare complex}
It is well-known (see Remark 3.2 in \cite[p. 245]{klim} and Corollary 1.2.3 in \cite[p. 8]{bro}) that for a Poincar\'{e} duality pair $(A, B)$ of dimension $n$ which is equipped with an orientation class $a \in H_{n}(A, B)$, the associated \emph{oriented boundary} $A = (A, \emptyset)$ is a Poincar\'{e} duality pair of dimension $n-1$ by means of the orientation class $\partial_{n} a \in H_{n-1}(A)$.

In the following, a Poincar\'{e} duality pair of the form $A = (A, \emptyset)$ will be called a \emph{Poincar\'{e} space}.

\end{remark}

We state our main technical result, which is a careful extension of Lemma 3.7 in \cite[p. 247]{klim}.
Our purpose is to cover also the case of depth one pseudomanifolds having non-isolated singular strata as considered in \cite{bc}.

\begin{proposition}\label{proposition poincare duality pairs}
Let $Z$ be a Poincar\'{e} space of dimension $n-1 > 0$ with orientation class $[Z] \in H_{n-1}(Z)$.
Let $f \colon Y \rightarrow Z$ be a map from a space $Y$ to $Z$ such that the induced map $f_{\ast} \colon H_{0}(Y) \rightarrow H_{0}(Z)$ is bijective, and let $X = \operatorname{cone}(f)$ denote the mapping cone of $f$.
Suppose that for every $r \in \mathbb{Z}$ the following conditions hold:
\begin{enumerate}[$(1)$]
\item The inclusion $Z \subset X$ induces a surjective homomorphism $H_{r}(Z) \rightarrow H_{r}(X)$.
\item The rational vector spaces $H_{r}(Y)$ and $\widetilde{H}_{n-r-1}(X)$ have the same rank.
\end{enumerate}
Then, for any space $X^{\phi} = \operatorname{cone}(\phi)$ given by the mapping cone of some map $\phi \colon S^{n-1} \rightarrow X$ the following statements are equivalent:
\begin{enumerate}[$(i)$]
\item There exists a lift $[e_{\phi}] \in H_{n}(X^{\phi}, Z)$ of the orientation class $[Z] \in H_{n-1}(Z)$ such that the pair  $(X^{\phi}, Z)$ is a Poincar\'{e} duality pair of dimension $n$ with orientation class $[e_{\phi}]$.
\item The orientation class $[Z] \in H_{n-1}(Z)$ of $Z$ lies in the image of the connecting homomorphism $\partial_{n} \colon H_{n}(X^{\phi}, Z) \rightarrow H_{n-1}(Z)$.
\item The image of the orientation class $[Z] \in H_{n-1}(Z)$ of $Z$ under the map $H_{n-1}(Z) \rightarrow H_{n-1}(X)$ induced by the inclusion $Z \subset X$ lies in the image of the connecting homomorphism $\partial_{n} \colon H_{n}(X^{\phi}, X) \rightarrow H_{n-1}(X)$.
\end{enumerate}
\end{proposition}

\begin{proof}
$(ii) \Leftrightarrow (iii)$.
The inclusion of pairs $(X^{\phi}, Z) \subset (X^{\phi}, X)$ induces the commutative diagram

\begin{center}
\begin{tikzpicture}

\path (0,0) node(a1) {$H_{n}(Z)$}
         (2.5, 0) node(a2) {$H_{n}(X^{\phi})$}
         (5, 0) node(a3) {$H_{n}(X^{\phi}, Z)$}
         (7.5, 0) node(a4) {$H_{n-1}(Z)$}
         (10, 0) node(a5) {$H_{n-1}(X^{\phi})$}
         (0,-1.5) node(b1) {$H_{n}(X)$}
         (2.5, -1.5) node(b2) {$H_{n}(X^{\phi})$}
         (5, -1.5) node(b3) {$H_{n}(X^{\phi}, X)$}
         (7.5, -1.5) node(b4) {$H_{n-1}(X)$}
         (10, -1.5) node(b5) {$H_{n-1}(X^{\phi})$.};

\draw[->] (a1) -- node[above] {} (a2);
\draw[->] (a2) -- node[above] {} (a3);
\draw[->] (a3) -- node[above] {$\partial_{n}$} (a4);
\draw[->] (a4) -- node[above] {} (a5);

\draw[->] (b1) -- node[above] {} (b2);
\draw[->] (b2) -- node[above] {} (b3);
\draw[->] (b3) -- node[above] {$\partial_{n}$} (b4);
\draw[->] (b4) -- node[above] {} (b5);

\draw[->] (a1) -- node[left] {$\eta$} (b1);
\draw[->] (a2) -- node[left] {$=$} (b2);
\draw[->] (a3) -- node[left] {$\zeta$} (b3);
\draw[->] (a4) -- node[left] {$\xi$} (b4);
\draw[->] (a5) -- node[left] {$=$} (b5);
\end{tikzpicture}
\end{center}
Let us show that $\eta$ and $\xi$ are isomorphisms.
First, note that $\eta$ and $\xi$ are surjective by assumption $(1)$ applied for $r = n$ and $r = n-1$, respectively.
Next, observe that $H_{n}(Z) \cong H_{-1}(Z) = 0$ and $H_{n-1}(Z) \cong H_{0}(Z)$ because $Z$ is a Poincar\'{e} space of dimension $n-1$.
Moreover, using that $H_{0}(Y) \cong H_{0}(Z)$ and that $n-1 > 0$, we obtain $H_{n-1}(X) = \widetilde{H}_{n-1}(X) \cong H_{0}(Y) \cong H_{0}(Z)$ by assumption $(2)$ applied for $r = 0$.
All in all, we have shown that $\eta$ and $\xi$ are isomorphisms.
Finally, the five lemma implies that the map $\zeta$ in the above diagram is an isomorphism as well, and the equivalence $(ii) \Leftrightarrow (iii)$ follows.

\begin{remark}\label{remark connecting homomorphisms}
For future reference we note that assuming either $(ii)$ or $(iii)$ in \Cref{proposition poincare duality pairs}, we can show that the connecting homomorphisms in the diagram

\begin{center}
\begin{tikzpicture}

\path (5, 0) node(a3) {$H_{n}(X^{\phi}, Z)$}
         (7.5, 0) node(a4) {$H_{n-1}(Z)$}
         (5, -1.5) node(b3) {$H_{n}(X^{\phi}, X)$}
         (7.5, -1.5) node(b4) {$H_{n-1}(X)$};

\draw[->] (a3) -- node[above] {$\partial_{n}$} (a4);
\draw[->] (b3) -- node[above] {$\partial_{n}$} (b4);
\draw[->] (a3) -- node[left] {$\cong$} (b3);
\draw[->] (a4) -- node[left] {$\cong$} (b4);
\end{tikzpicture}
\end{center}
are both injective because $H_{n}(X^{\phi}, Z) \cong H_{n}(X^{\phi}, X) \cong \mathbb{Q}$, and they are non-trivial.
(Indeed, observe that we have
$$
H_{n}(X^{\phi}, X) \cong \widetilde{H}_{n}(X^{\phi}/ X) \cong \widetilde{H}_{n}(S^{n}) \cong \mathbb{Q}.
$$
Here, the isomorphism $H_{n}(X^{\phi}, X) \cong \widetilde{H}_{n}(X^{\phi}/ X)$ holds by Proposition 2.22 in \cite[p. 124]{hat}.
Note that for any map $\alpha \colon C \rightarrow D$, the pair $(\operatorname{cone}(\alpha), D)$ is a \emph{good pair}, that is, $D$ is a nonempty closed subset of $\operatorname{cone}(\alpha)$ that is a deformation retract of some neighborhood in $\operatorname{cone}(\alpha)$.)
\end{remark}

$(i) \Rightarrow (ii)$.
This implication is clear because statement $(i)$ implies that $[Z] = \partial_{n}([e_{\phi}])$.

$(ii) \Rightarrow (i)$.
By statement $(ii)$ there exists an element $[e_{\phi}] \in H_{n}(X^{\phi}, Z)$ such that $\partial_{n}([e_{\phi}]) = [Z] \in H_{n-1}(Z)$.
(In fact, $e_{\phi}$ is uniquely determined because $\partial_{n}$ is injective by \Cref{remark connecting homomorphisms}.)
We claim that the homomorphism
$$
- \cap [e_{\phi}] \colon H^{r}(X^{\phi}) \rightarrow H_{n-r}(X^{\phi}, Z)
$$
is an isomorphism for every $r \in \mathbb{Z}$.
(It is clear by assumption $(1)$ that all rational homology groups of $X^{\phi}$ have finite rank, so that $(X^{\phi}, Z)$ will then be a Poincar\'{e} duality pair of dimension $n$ according to \Cref{definition poincare duality pair}.)
In general, observe that every element $\alpha \in H_{n}(X^{\phi}, Z)$ gives by Proposition 1.1.4(ii) in \cite[p. 4]{bro} rise to a commutative diagram

\begin{center}
\begin{tikzpicture}

\path (0,0) node(a) {$H^{r}(X^{\phi})$}
         (4, 0) node(b) {$H_{n-r}(X^{\phi}, Z)$}
         (0, -2) node(c) {$H^{r}\left(Z\right)$}
         (4, -2) node(d) {$H_{n-1-r}\left(Z\right)$.};

\draw[->] (a) -- node[left] {$\operatorname{incl}^{\ast}$} (c);
\draw[->] (a) -- node[above] {$- \cap \alpha$} (b);
\draw[->] (c) -- node[above] {$- \cap \partial_{n} \alpha$} (d);
\draw[->] (b) -- node[right] {$\partial_{n-r}$} (d);
\end{tikzpicture}
\end{center}
If we specialize to $\alpha = [e_{\phi}]$, then the lower horizontal homomorphism $- \cap [Z]$ is an isomorphism because $Z$ is a Poincar\'{e} space of dimension $n-1$.
Thus, in order to show that the upper horizontal row of the above diagram is an isomorphism, it suffices to verify that for every $r \in \mathbb{Z}$, the following two assertions hold:
\begin{enumerate}[$(a)$]
\item \emph{The inclusion $Z \subset X^{\phi}$ induces a surjective map $H_{r}\left(Z\right) \rightarrow H_{r}(X^{\phi})$.}

In view of assumption $(1)$, our claim $(a)$ is in fact equivalent to showing that the inclusion $X \subset X^{\phi}$ induces a surjective homomorphism $H_{r}\left(X\right) \rightarrow H_{r}(X^{\phi})$.
We compute
$$
H_{r}(X^{\phi}, X) \cong \widetilde{H}_{r}(X^{\phi}/ X) \cong \widetilde{H}_{r}(S^{n}) \cong \begin{cases}
\mathbb{Q}, \quad r = n, \\
0, \quad r \neq n.
\end{cases}
$$
Thus, for $r \neq n$, the claim follows from the exactness of
$$
H_{r}(X) \rightarrow H_{r}(X^{\phi}) \rightarrow H_{r}(X^{\phi}, X).
$$
For $r = n$ we consider the exact sequence
$$
H_{n}(X) \rightarrow H_{n}(X^{\phi}) \rightarrow H_{n}(X^{\phi}, X) \stackrel{\partial_{n}}{\rightarrow} H_{n-1}(X).
$$
Since the connecting homomorphism $\partial_{n}$ is injective by \Cref{remark connecting homomorphisms}, we conclude that the homomorphism $H_{n}(X) \rightarrow H_{n}(X^{\phi})$ is surjective.

\begin{remark}\label{remark homology calculations}
Suppose that either $(ii)$ or $(iii)$ holds in \Cref{proposition poincare duality pairs}.
Then, for future reference, we note that an inspection of the homology long exact sequence of the pair $(X^{\phi}, X)$ combined with the above information implies that the inclusion $X \subset X^{\phi}$ induces an isomorphism $H_{r}(X) \stackrel{\cong}{\longrightarrow} H_{r}(X^{\phi})$ for $r \neq n-1$, whereas the induced map $H_{n-1}(X) \rightarrow H_{n-1}(X^{\phi})$ is surjective with kernel isomorphic to $H_{n}(X^{\phi}, X) \cong \mathbb{Q}$.
\end{remark}

\item \emph{The rational vector spaces $H_{r}(X^{\phi})$ and $H_{n-r}(X^{\phi}, Z)$ have the same (finite) rank.}

In view of assumption $(2)$ and \Cref{remark homology calculations}, it suffices to show that
$$
H_{r}(X^{\phi}, Z) \cong \begin{cases}
\widetilde{H}_{r-1}(Y), \quad r \neq n, \\
H_{n-1}(Y) \oplus \mathbb{Q}, \quad r = n.
\end{cases}
$$
The inclusion of pairs $(X, Z) \subset (X^{\phi}, Z)$ induces the commutative diagram

\begin{center}
\begin{tikzpicture}

\path (0,0) node(a1) {$H_{r}(Z)$}
         (2.5, 0) node(a2) {$H_{r}(X)$}
         (5, 0) node(a3) {$H_{r}(X, Z)$}
         (7.5, 0) node(a4) {$H_{r-1}(Z)$}
         (10, 0) node(a5) {$H_{r-1}(X)$}
         (0,-1.5) node(b1) {$H_{r}(Z)$}
         (2.5, -1.5) node(b2) {$H_{r}(X^{\phi})$}
         (5, -1.5) node(b3) {$H_{r}(X^{\phi}, Z)$}
         (7.5, -1.5) node(b4) {$H_{r-1}(Z)$}
         (10, -1.5) node(b5) {$H_{r-1}(X^{\phi})$.};

\draw[->] (a1) -- node[above] {} (a2);
\draw[->] (a2) -- node[above] {} (a3);
\draw[->] (a3) -- node[above] {$\partial_{r}$} (a4);
\draw[->] (a4) -- node[above] {} (a5);

\draw[->] (b1) -- node[above] {} (b2);
\draw[->] (b2) -- node[above] {} (b3);
\draw[->] (b3) -- node[above] {$\partial_{r}$} (b4);
\draw[->] (b4) -- node[above] {} (b5);

\draw[->] (a1) -- node[left] {$=$} (b1);
\draw[->] (a2) -- node[left] {$$} (b2);
\draw[->] (a3) -- node[left] {$$} (b3);
\draw[->] (a4) -- node[left] {$=$} (b4);
\draw[->] (a5) -- node[left] {$$} (b5);
\end{tikzpicture}
\end{center}
Observe that for all $r \in \mathbb{Z}$,
$$
H_{r}(X, Z) \cong \widetilde{H}_{r}(X/Z) \cong \widetilde{H}_{r}(\Sigma Y) \cong \widetilde{H}_{r-1}(Y).
$$
(Here, the isomorphism $H_{r}(X, Z) \cong \widetilde{H}_{r}(X/Z)$ holds by Proposition 2.22 in \cite[p. 124]{hat} because $X$ is the cone of $f \colon Y \rightarrow Z$.)
Since by \Cref{remark homology calculations} the inclusion $X \subset X^{\phi}$ induces an isomorphism $H_{r}(X) \stackrel{\cong}{\longrightarrow} H_{r}(X^{\phi})$ for $r \neq n-1$, the claim follows for $r \notin \{n-1, n\}$ from the five lemma applied to the above diagram.
We check the remaining cases:
\begin{itemize}
\item In the case $r = n-1$, we note that in the above diagram the homomorphism $H_{n-1}(Z) \rightarrow H_{n-1}(X)$ is surjective by assumption $(1)$, and the homomorphism $H_{n-1}(X) \rightarrow H_{n-1}(X^{\phi})$ is surjective by \Cref{remark homology calculations}.
Thus, we obtain the following simplified commutative diagram with exact rows:
\begin{center}
\begin{tikzpicture}

\path (0,0) node(a1) {$0$}
         (1.5, 0) node(a2) {$0$}
         (4, 0) node(a3) {$H_{n-1}(X, Z)$}
         (6.5, 0) node(a4) {$H_{n-2}(Z)$}
         (9, 0) node(a5) {$H_{n-2}(X)$}
         (0,-1.5) node(b1) {$0$}
         (1.5, -1.5) node(b2) {$0$}
         (4, -1.5) node(b3) {$H_{n-1}(X^{\phi}, Z)$}
         (6.5, -1.5) node(b4) {$H_{n-2}(Z)$}
         (9, -1.5) node(b5) {$H_{n-2}(X^{\phi})$.};

\draw[->] (a1) -- node[above] {} (a2);
\draw[->] (a2) -- node[above] {} (a3);
\draw[->] (a3) -- node[above] {$\partial_{n-1}$} (a4);
\draw[->] (a4) -- node[above] {} (a5);

\draw[->] (b1) -- node[above] {} (b2);
\draw[->] (b2) -- node[above] {} (b3);
\draw[->] (b3) -- node[above] {$\partial_{n-1}$} (b4);
\draw[->] (b4) -- node[above] {} (b5);

\draw[->] (a1) -- node[left] {$$} (b1);
\draw[->] (a2) -- node[left] {$$} (b2);
\draw[->] (a3) -- node[left] {$$} (b3);
\draw[->] (a4) -- node[left] {$=$} (b4);
\draw[->] (a5) -- node[left] {$$} (b5);
\end{tikzpicture}
\end{center}
Finally, as the right vertical arrow $H_{n-2}(X) \rightarrow H_{n-2}(X^{\phi})$ is an isomorphism by \Cref{remark homology calculations}, the five lemma yields
$$
H_{n-1}(X^{\phi}, Z) \cong H_{n-1}(X, Z) \cong \widetilde{H}_{n-2}(Y).
$$
\item In the case $r = n$, we consider the following portion of the homology exact sequence of the pair $(X, Z)$:
$$
H_{n}(X) \rightarrow H_{n}(X, Z) \rightarrow H_{n-1}(Z) \rightarrow H_{n-1}(X).
$$
Note that $H_{n-1}(X) \cong H_{0}(Y)$ and $H_{n}(X) \cong 0$ by assumption $(2)$. As the homomorphism $H_{n-1}(Z) \rightarrow H_{n-1}(X)$ is surjective by assumption $(1)$, we obtain
$$
H_{n-1}(Z) \cong H_{n}(X, Z) \oplus H_{n-1}(X) \cong H_{n-1}(Y) \oplus H_{0}(Y).
$$
Hence, it follows from $H_{n-1}(Z) \cong H_{0}(Z) \cong H_{0}(Y)$ that $H_{n-1}(Y) = 0$.
On the other hand, $H_{n}(X^{\phi}, Z) \cong \mathbb{Q}$ according to \Cref{remark connecting homomorphisms}.
\end{itemize}
\begin{remark}\label{remark isomorphisms induced by inclusions}
Suppose that either $(ii)$ or $(iii)$ holds in \Cref{proposition poincare duality pairs}.
Then, for future reference, we note that the inclusion of pairs $(X, Z) \subset (X^{\phi}, Z)$ induces an isomorphism $H_{r}(X, Z) \stackrel{\cong}{\longrightarrow} H_{r}(X^{\phi}, Z)$ for $r \neq n$, whereas $H_{n}(X^{\phi}, Z) \cong H_{n}(X, Z) \oplus \mathbb{Q}$.
\end{remark}
\end{enumerate}
\end{proof}

\subsection{Proof of \Cref{main theorem}}\label{proof of main theorem}
Following Definition 6.2 in \cite{bc}, we define the fiberwise cotruncation $\operatorname{ft}_{\geq} E$ (both of degree $k$ and of degree $l$) of the fiber bundle $\pi \colon E \rightarrow B$ as the homotopy pushout of the diagram
$$
B \stackrel{\pi_{<}}{\longleftarrow} \operatorname{ft}_{<} E \stackrel{F_{<}}{\longrightarrow} E.
$$
In Definition 6.3 of \cite{bc} an auxiliary space $Q_{\geq}E$ together with a structure map $C_{\geq} \colon Q_{\geq}E \rightarrow E$ is introduced.
In Section 6 of \cite{bc} the following properties of $Q_{\geq}E$ are derived (the assumption that $B$ admits a good cover is needed to deduce item (a) and (c) because they are both based on Proposition 6.5 in \cite{bc}):

\begin{enumerate}[(a)]
\item By Lemma 6.6 in \cite{bc}, the map $C_{\geq} \colon E \rightarrow Q_{\geq}E$ induces for every $r \in \mathbb{Z}$ a surjective homomorphism $C_{\geq \ast} \colon H_{r}(E) \rightarrow H_{r}(Q_{\geq}E)$.

\item By the proof of Lemma 6.6 in \cite{bc}, there exists a diagram
\begin{center}
\begin{tikzpicture}

\path (0,0) node(a) {$E$}
         (4, 0) node(b) {$\operatorname{cone}(F_{<})$}
         (0, -2) node(c) {$Q_{\geq E}$}
         (4, -2) node(d) {$\operatorname{ft}_{\geq} E/B$};

\draw[->] (a) -- node[left] {$C_{\geq}$} (c);
\draw[->] (a) -- node[above] {$\operatorname{incl}$} (b);
\draw[->] (c) -- node[above] {$\simeq$} (d);
\draw[->] (b) -- node[right] {$\simeq$} (d);
\end{tikzpicture}
\end{center}
which commutes on the nose.
In particular, $Q_{\geq E}$ and $\operatorname{cone}(F_{<})$ are homotopy equivalent spaces.

\item According to Proposition 6.7 in \cite{bc}, the rational vector spaces $H_{r}(\operatorname{ft}_{<} E)$ and $\widetilde{H}_{n-r-1}(Q_{\geq}E)$ have for every $r \in \mathbb{Z}$ the same rank.
\end{enumerate}

Let us prove the equivalence $(i) \Leftrightarrow (ii)$ claimed by \Cref{main theorem}.
For this purpose, we first apply the equivalence $(i) \Leftrightarrow (iii)$ of \Cref{proposition poincare duality pairs} to the map $f \colon Y \rightarrow Z$ given by $F_{<} \colon \operatorname{ft}_{<} E \rightarrow E$.
Note that the space $Z = E$ is a Poincar\'{e} space with orientation class given by $[Z] = [E] \in H_{n-1}(Z)$ because $E$ is a closed oriented manifold.
Moreover, note that the map $F_{<} \colon \operatorname{ft}_{<} E \rightarrow E$ induces a bijection $F_{<} \colon H_{0}(\operatorname{ft}_{<} E) \rightarrow H_{0}(E)$.
Taking $X = \operatorname{cone}(F_{<})$ to be the mapping cone of the map $f = F_{<}$, we see the condition (1) of \Cref{proposition poincare duality pairs} holds by properties (a) and (b) of $Q_{\geq E}$, and condition (2) of \Cref{proposition poincare duality pairs} holds by properties (b) and (c) of $Q_{\geq E}$.
Let $x \in H_{n-1}(X)$ denote the image of the orientation class $[Z] \in H_{n-1}(Z)$ of $Z$ under the map $H_{n-1}(Z) \rightarrow H_{n-1}(X)$ induced by the inclusion $Z \subset X$.
Then, by using the equivalence $(i) \Leftrightarrow (iii)$ of \Cref{proposition poincare duality pairs}, we conclude that statement $(i)$ of \Cref{main theorem} is equivalent to the existence of a map $\phi \colon (S^{n-1}, s_{0}) \rightarrow (\operatorname{cone}(F_{<}), \operatorname{pt}) = (X, \operatorname{pt})$ such that the element $x \in H_{n-1}(X)$ is contained in the image of the connecting homomorphism $\partial_{n} \colon H_{n}(\operatorname{cone}(\phi), X) \rightarrow H_{n-1}(X)$.
Finally, \Cref{lemma hurewicz map} implies that the latter statement is equivalent to statement $(iii)$ of \Cref{main theorem}.
This completes the proof of the equivalence $(i) \Leftrightarrow (ii)$. \\

We continue to use the notation $f \colon Y \rightarrow Z$ for the map $F_{<} \colon \operatorname{ft}_{<} E \rightarrow E$, and write $X = \operatorname{cone}(F_{<})$ for the mapping cone of $f = F_{<}$.
From now on we suppose that either of the equivalent statements of \Cref{main theorem} holds.
Fix $r \in \{1, \dots, n-2\}$.
We have to show that the local duality obstruction $\mathcal{O}_{n-1-r}(\pi, k, l)$ vanishes.
By statement $(i)$ there is a map $\phi \colon S^{n-1} \rightarrow X$ and a lift $[e_{\phi}] \in H_{n}(X^{\phi}, Z)$ of the orientation class $[Z] \in H_{n-1}(Z)$ such that $(X^{\phi}, Z)$ is a rational Poincar\'{e} duality pair of dimension $n$ with orientation class $[e_{\phi}]$.
By Theorem I.2.2 in \cite[p. 8]{bro}, there is the following commutative diagram:
\begin{center}
\begin{tikzpicture}

\path (0,0) node(a) {$H^{r}(X^{\phi})$}
         (4, 0) node(b) {$H_{n-r}(X^{\phi}, Z)$}
         (0, -2) node(c) {$H^{r}\left(Z\right)$}
         (4, -2) node(d) {$H_{n-1-r}\left(Z\right)$.};

\draw[->] (a) -- node[left] {$\operatorname{incl}^{\ast}$} (c);
\draw[->] (a) -- node[above] {$-\cap [e_{\phi}]$} (b);
\draw[->] (c) -- node[above] {$- \cap [Z]$} (d);
\draw[->] (b) -- node[right] {$\partial_{n-r}$} (d);
\end{tikzpicture}
\end{center}

Factoring the inclusion $Z \subset X^{\phi}$ as $Z \subset X \subset X^{\phi}$, and using the inclusion of pairs $(X, Z) \subset (X^{\phi}, Z)$, we can extend the diagram to a commutative diagram as follows:

\begin{center}
\begin{tikzpicture}

\path (0,0) node(a) {$H^{r}(X^{\phi})$}
         (3, 0) node(b) {$H_{n-r}(X^{\phi}, Z)$}
         (0, -2) node(c) {$H^{r}\left(Z\right)$}
         (3, -2) node(d) {$H_{n-1-r}\left(Z\right)$.}
         (-3, 0) node(e) {$H^{r}(X)$}
         (6, 0) node(f) {$H_{n-r}(X, Z)$};

\draw[->] (a) -- node[left] {$\operatorname{incl}^{\ast}$} (c);
\draw[->] (a) -- node[above] {$-\cap [e_{\phi}]$} (b);
\draw[->] (c) -- node[above] {$- \cap [Z]$} (d);
\draw[->] (b) -- node[right] {$\partial_{n-r}$} (d);
\draw[->] (a) -- node[above] {$\operatorname{incl}^{\ast}$} (e);
\draw[->] (e) -- node[left] {$\operatorname{incl}^{\ast}$} (c);
\draw[->] (f) -- node[above] {$\operatorname{incl}^{\ast}$} (b);
\draw[->] (f) -- node[right] {$\partial_{n-r}$} (d);
\end{tikzpicture}
\end{center}
Since for $r \in \{1, \dots, n-2\}$ the inclusions $X \subset X^{\phi}$ and $(X, Z) \subset (X^{\phi}, Z)$ induce isomorphisms $H^{r}(X^{\phi}) \stackrel{\cong}{\longrightarrow} H^{r}(X)$ by \Cref{remark homology calculations} and $H_{n-r}(X, Z) \stackrel{\cong}{\longrightarrow} H_{n-r}(X^{\phi}, Z)$ by \Cref{remark isomorphisms induced by inclusions}, respectively, we obtain a commutative diagram of the form

\begin{center}
\begin{tikzpicture}

\path (0,0) node(a) {$H^{r}(X)$}
         (4, 0) node(b) {$H_{n-r}(X, Z)$}
         (0, -2) node(c) {$H^{r}\left(Z\right)$}
         (4, -2) node(d) {$H_{n-1-r}\left(Z\right)$.};

\draw[->] (a) -- node[left] {$\operatorname{incl}^{\ast}$} (c);
\draw[->] (a) -- node[above] {$D_{r}^{\phi} \quad \cong$} (b);
\draw[->] (c) -- node[above] {$- \cap [Z]$} (d);
\draw[->] (b) -- node[right] {$\partial_{n-r}$} (d);
\end{tikzpicture}
\end{center}
We claim that the previous diagram is part of a commutative diagram of the form

\begin{center}
\begin{tikzpicture}

\path (0,0) node(a) {$H^{r}(X)$}
         (3, 0) node(b) {$H_{n-r}(X, Z)$}
         (0, -2) node(c) {$H^{r}\left(Z\right)$}
         (3, -2) node(d) {$H_{n-1-r}\left(Z\right)$}
         (-3, 0) node(e) {$H^{r}(\operatorname{ft}_{\geq}E/B)$}
         (6, 0) node(f) {$H_{n-r}(\operatorname{cone}(Y), Y)$}
         (-3, -2) node(g) {$H^{r}(Q_{\geq}E)$}
         (6, -2) node(h) {$H_{n-1-r}(Y)$.}
         (4.5, -1) node(i) {$(II)$}
         (-1.5, -1) node(i) {$(I)$};

\draw[->] (a) -- node[right] {$\operatorname{incl}^{\ast}$} (c);
\draw[->] (a) -- node[above] {$D_{r}^{\phi} \; \cong$} (b);
\draw[->] (c) -- node[above] {$- \cap [Z]$} (d);
\draw[->] (b) -- node[left] {$\partial_{n-r}$} (d);
\draw[->] (h) -- node[above] {$f_{\ast}$} (d);
\draw[->] (f) -- node[right] {$\cong$} (h);
\draw[->] (f) -- node[above] {$\cong$} (b);
\draw[->] (e) -- node[above] {$\cong$} (a);
\draw[->] (e) -- node[left] {$\cong$} (g);
\draw[->] (g) -- node[above] {$C_{\geq}^{\ast}$} (c);
\end{tikzpicture}
\end{center}
In fact, the commutative square $(I)$ exists by property $(b)$ of $Q_{\geq} E$, and the commutative square $(II)$ is given by the square $(II)'$ in the following commutative diagram, which is induced by the canonical map of pairs $(\widetilde{f}, f) \colon (\operatorname{cone}(Y), Y) \rightarrow (X, Z)$:

\begin{center}
\begin{tikzpicture}

\path (0,0) node(a) {$H_{n-r}(\operatorname{cone}(Y), Y)$}
         (4.5, 0) node(b) {$\widetilde{H}_{n-r}(\operatorname{cone}(Y)/Y)$}
         (0, -2) node(c) {$H_{n-r}\left(X, Z\right)$}
         (4.5, -2) node(d) {$\widetilde{H}_{n-r}\left(X/Z\right)$.}
         (-5, 0) node(e) {$H_{n-r-1}(Y)$}
         (-2.5, -1) node(l) {$(II)'$}
         (-5, -2) node(g) {$H_{n-r-1}(Z)$};

\draw[->] (a) -- node[right] {$\widetilde{f}_{\ast}$} (c);
\draw[->] (a) -- node[above] {$\operatorname{quot}_{\ast}$} (b);
\draw[->] (c) -- node[above] {$\operatorname{quot}_{\ast}$} (d);
\draw[->] (b) -- node[left] {$\cong$} (d);
\draw[->] (a) -- node[above] {$\partial_{n-r}$} (e);
\draw[->] (e) -- node[left] {$f_{\ast}$} (g);
\draw[->] (c) -- node[above] {$\partial_{n-r}$} (g);
\end{tikzpicture}
\end{center}
(Note that $\widetilde{f}_{\ast}$ is an isomorphism because the horizontal arrows induced by the quotient maps $(\operatorname{cone}(Y), Y) \rightarrow (\operatorname{cone}(Y)/Y, \operatorname{pt})$ and $(X, Z) \rightarrow (X/Z, \operatorname{pt})$ are isomorphisms (see Proposition 2.22 in \cite[p. 124]{hat}, and note that $(\operatorname{cone}(Y), Y)$ and $(X, Z)$ are good pairs).
Moreover, the connecting homomorphism
$$
\partial_{n-r} \colon H_{n-r}(\operatorname{cone}(Y), Y) \rightarrow H_{n-r-1}(Y)
$$
is an isomorphism because $\operatorname{cone}(Y)$ is contractible, and $r \neq n-1$.)

All in all, we obtain a commutative diagram
\begin{center}
\begin{tikzpicture}

\path (0,0) node(a) {$H^{r}(Q_{\geq}E)$}
         (4, 0) node(b) {$H_{n-r-1}(Y)$}
         (0, -2) node(c) {$H^{r}\left(Z\right)$}
         (4, -2) node(d) {$H_{n-1-r}\left(Z\right)$.};

\draw[->] (a) -- node[left] {$C_{\geq}^{\ast}$} (c);
\draw[->] (a) -- node[above] {$\cong$} (b);
\draw[->] (c) -- node[above] {$- \cap [Z]$} (d);
\draw[->] (b) -- node[right] {$f_{\ast}$} (d);
\end{tikzpicture}
\end{center}
After applying the functor $\operatorname{Hom}_{\mathbb{Q}}(-, \mathbb{Q})$, we conclude from Proposition 6.9 in \cite{bc} that $\mathcal{O}_{n-1-r}(\pi, k, l)$ vanishes.
(Moreover, by the same proposition, the isomorphism $H^{r}(Q_{\geq}E) \rightarrow H_{n-r-1}(Y)$ is uniquely determined by commutativity of the above diagram.
Consequently, the isomorphism $D_{r}^{\phi}$ does actually not depend on $\phi$, and is thus the desired isomorphism $D_{r}^{\phi} = D_{r}$ used in \Cref{remark independence of attaching map}.)

This completes the proof of \Cref{main theorem}.

\subsection{Connectivity of truncation cones}\label{Connectivity of truncation cones}
In order to be able to apply rational homotopy theory to truncation cones in \Cref{rational homotopy theory}, we study the connectivity properties of truncation cones in the present section.

The following result shows that the reduced homology of truncation cones vanishes in low degrees.

\begin{lemma}\label{lemma homology of truncation cone}
If $B$ admits a good cover (e.g., if $B$ is smooth or at least PL), then the homology of $\operatorname{cone}(F_{<})$ satisfies $\widetilde{H}_{i}(\operatorname{cone}(F_{<})) = 0$ for $i < \operatorname{max}\{k, l\}$.
\end{lemma}

\begin{proof}
We employ the local to global technique based on precosheaves as presented in Section 4 of \cite{bc}, and assume familiarity with the notation and definitions used therein.

Our argument requires a slight modification of Proposition 4.4 of \cite{bc} that applies to \emph{finite} sequences of $\delta$-compatible morphisms between precosheaves instead of inifinite sequences (see \Cref{proposition local to global via precosheaves} below).
For this purpose, consider an open cover $\mathcal{U}$ of the topological space $B$.
Let $\tau \mathcal{U}$ denote the category whose objects are unions of finite intersections of open sets in $\mathcal{U}$, and whose morphisms are inclusions.
On the product category $\tau \mathcal{U} \times \tau \mathcal{U}$, consider the functors $\cap, \cup \colon \tau \mathcal{U} \times \tau \mathcal{U} \rightarrow \tau \mathcal{U}$ induced by intersection and union of open sets, respectively.
For $i = 1, 2$, projection to the $i$-th factor determines a projection functor $p_{i} \colon \tau \mathcal{U} \times \tau \mathcal{U} \rightarrow \tau \mathcal{U}$.
We also need for $i = 1, 2$ the natural transformations $j_{i} \colon p_{i} \rightarrow \cup$ and $\iota_{i} \colon \cap \rightarrow p_{i}$ induced by the inclusions $U, V \subset U \cup V$ and $U \cap V \subset U, V$, respectively.
Now consider a finite sequence $\mathcal{F}_{i} = \mathcal{F}_{0}, \dots, \mathcal{F}_{N}$ of precosheaves on $B$.
Slightly modifying Definition 4.3 in \cite{bc}, we say that the finite sequence $\mathcal{F}_{i}$ satisfies the \emph{$\mathcal{U}$-Mayer-Vietoris property} if there are natural transformations of functors on $\tau \mathcal{U} \times \tau \mathcal{U}$,
$$
\delta_{i}^{\mathcal{F}} \colon \mathcal{F}_{i} \circ \cup \rightarrow \mathcal{F}_{i-1} \circ \cap, \qquad 1 \leq i \leq N,
$$
such that for every pair of open sets $U, V \in \tau \mathcal{U}$ the following sequence is exact:
\begin{align*}
&\mathcal{F}_{N}(U \cap V) \stackrel{(\iota_{1}^{N}, \iota_{2}^{N})}{\longrightarrow} \mathcal{F}_{N}(U) \oplus \mathcal{F}_{N}(V) \stackrel{j_{1}^{N}-j_{2}^{N}}{\longrightarrow} \mathcal{F}_{N}(U \cup V)
\stackrel{\delta_{N}^{\mathcal{F}}}{\longrightarrow} \dots \\
\dots \stackrel{\delta_{i+1}^{\mathcal{F}}}{\longrightarrow} &\mathcal{F}_{i}(U \cap V) \stackrel{(\iota_{1}^{i}, \iota_{2}^{i})}{\longrightarrow} \mathcal{F}_{i}(U) \oplus \mathcal{F}_{i}(V) \stackrel{j_{1}^{i}-j_{2}^{i}}{\longrightarrow} \mathcal{F}_{i}(U \cup V) \stackrel{\delta_{i}^{\mathcal{F}}}{\longrightarrow} \dots \\
\dots \stackrel{\delta_{1}^{\mathcal{F}}}{\longrightarrow} &\mathcal{F}_{0}(U \cap V) \stackrel{(\iota_{1}^{0}, \iota_{2}^{0})}{\longrightarrow} \mathcal{F}_{0}(U) \oplus \mathcal{F}_{0}(V) \stackrel{j_{1}^{0}-j_{2}^{0}}{\longrightarrow} \mathcal{F}_{0}(U \cup V) \longrightarrow 0.
\end{align*}
(Note that the only difference to Definition 4.3 in \cite{bc} is that our sequence ends to the left with the term $\mathcal{F}_{N}(U \cap V)$ because $\mathcal{F}_{N+1}(U \cup V)$, $\delta_{N+1}^{\mathcal{F}}$, etc. are not defined.)
In accordance with Definition 4.3 in \cite{bc}, a collection of morphisms of precosheaves $f_{i} \colon \mathcal{F}_{i} \rightarrow \mathcal{G}_{i}$, $0 \leq i \leq N$, is called \emph{$\delta$-compatible} if for every pair of open sets $U, V \in \tau \mathcal{U}$ the following diagram commutes for all $i$ (where $0 \leq i \leq N-1$ in our setting):
\begin{center}
\begin{tikzpicture}

\path (0,0) node(a) {$\mathcal{F}_{i+1}(U \cup V)$}
         (4, 0) node(b) {$\mathcal{F}_{i}(U \cap V)$}
         (0, -2) node(c) {$\mathcal{G}_{i+1}(U \cup V)$}
         (4, -2) node(d) {$\mathcal{G}_{i}(U \cap V)$.};

\draw[->] (a) -- node[left] {$f_{i+1}(U \cup V)$} (c);
\draw[->] (a) -- node[above] {$\delta_{i+1}^{\mathcal{F}}(U, V)$} (b);
\draw[->] (b) -- node[right] {$f_{i}(U \cap V)$} (d);
\draw[->] (c) -- node[above] {$\delta_{i+1}^{\mathcal{G}}(U, V)$} (d);
\end{tikzpicture}
\end{center}
Next we state our adaption of Proposition 4.4 in \cite{bc} to finite sequences of $\delta$-compatible morphisms between precosheaves.
(The only change in the proof is that the five lemma is applied to a commutative ladder that ends on the left with the homomorphism $f_{N}(V^{j}) \colon \mathcal{F}_{N}(V^{j}) \rightarrow \mathcal{G}_{N}(V^{j})$.)
\begin{proposition}\label{proposition local to global via precosheaves}
Let $\mathcal{U}$ be an open cover of the topological space $B$.
Let $f_{i} \colon \mathcal{F}_{i} \rightarrow \mathcal{G}_{i}$, $0 \leq i \leq N$, be a finite sequence of $\delta$-compatible morphisms between $\mathcal{U}$-locally constant precosheaves on $B$ that satisfy the $\mathcal{U}$-Mayer-Vietoris property.
If the map $f_{i}(U) \colon \mathcal{F}_{i}(U) \rightarrow \mathcal{G}_{i}(U)$ is an isomorphism for every $U \in \mathcal{U}$ and for every $0 \leq i \leq N$, then $f_{i}(B) \colon \mathcal{F}_{i}(B) \rightarrow \mathcal{G}_{i}(B)$ is an isomorphism for all $0 \leq i \leq N$.
\end{proposition}

Recall that we have a fiber bundle $\pi_{<} : \operatorname{ft}_{<}E \rightarrow B$ with fiver $L_{<}$.
Consider also the fiber bundle $\rho \colon \operatorname{cyl}(F_{<}) \rightarrow B$ with fiber $\operatorname{cyl}(f_{<})$, and note that there are morphisms of fiber bundles $\operatorname{ft}_{<}E \rightarrow \operatorname{cyl}(F_{<}) \rightarrow E$ restricting to fiberwise maps $L_{<} \rightarrow \operatorname{cyl}(f_{<}) \rightarrow L$.

Fix a good cover $\mathcal{U}$ of $B$, and set $N = \operatorname{max}\{k, l\}-1$.
We apply \Cref{proposition local to global via precosheaves} to the precosheaves $\mathcal{F}_{i}$ and $\mathcal{G}_{i}$, $0 \leq i \leq N$, given on open sets $U \in \tau \mathcal{U}$ by $\mathcal{F}_{i}(U) = H_{i}(\rho^{-1}(U), \pi_{<}^{-1}(U))$ and $\mathcal{G}_{i}(U) = 0$, and the morphisms of precosheaves $f_{i} \colon \mathcal{F}_{i} \rightarrow \mathcal{G}_{i}$ determined by the unique map $f_{i}(U) \colon H_{i}(\rho^{-1}(U), \pi_{<}^{-1}(U)) \rightarrow 0$ for open sets $U \in \tau \mathcal{U}$.
Note that, by the Eilenberg-Steenrod axioms, the precosheaves $\mathcal{F}_{i}$ are $\mathcal{U}$-locally constant because $\mathcal{U}$ is a good open cover of $B$, and satisfy the $\mathcal{U}$-Mayer-Vietoris property with respect to the connecting homomorphisms of the relative form of the Mayer-Vietoris sequence. 
It is clear that the trivial precosheaves $\mathcal{G}_{i}$ are $\mathcal{U}$-locally constant and satisfy the $\mathcal{U}$-Mayer-Vietoris property with respect to the natural transformations $\delta_{i}^{\mathcal{G}} = 0$.
Obviously, the morphisms of precosheaves $f_{i} \colon \mathcal{F}_{i} \rightarrow \mathcal{G}_{i}$ are $\delta$-compatible.
Moreover, as every open set $U \in \mathcal{U}$ is contractible, we conclude for every $0 \leq i \leq N$ ($< \operatorname{max}\{k, l\}$) that $f_{i}(U) \colon \mathcal{F}_{i}(U) \rightarrow \mathcal{G}_{i}(U) = 0$ is an isomorphism because
$$
\mathcal{F}_{i}(U) = H_{i}(\rho^{-1}(U), \pi_{<}^{-1}(U)) \cong H_{i}(\operatorname{cyl}(f_{<}), L_{<}) \cong \widetilde{H}_{i}(\operatorname{cone}(f_{<})) = 0.
$$
Hence, \Cref{proposition local to global via precosheaves} implies that $\mathcal{F}_{i}(B) = 0$ for all $0 \leq i \leq N$, where note that
$$
\mathcal{F}_{i}(B) = H_{i}(\operatorname{cyl}(F_{<}), \operatorname{ft}_{<}E) \cong \widetilde{H}_{i}(\operatorname{cone}(F_{<})).
$$
\end{proof}

The following result provides a proof of simple connectivity of truncation cones, at least under a mild hypothesis on the underlying Moore approximation.
Simple connectivity of the truncation cone will enable us to invoke the machinery of rational homotopy theory in \Cref{rational homotopy theory}.

\begin{lemma}\label{lemma simply connected cone}
The truncation cone, $\operatorname{cone}(F_{<})$, is path connected.
If the equivariant Moore approximation $f_{<} \colon L_{<} \rightarrow L$ induces a surjection $f_{< \ast} \colon \pi_{1}(L_{<}, x_{0}) \rightarrow \pi_{1}(L, f_{<}(x_{0}))$ for every basepoint $x_{0} \in L_{<}$, then $\operatorname{cone}(F_{<})$ is simply connected.
\end{lemma}

\begin{proof}
Since $f_{<} \colon L_{<} \rightarrow L$ is a Moore approximation of positive degree, the induced map $H_{0}(L_{<}) \rightarrow H_{0}(L)$ is an isomorphism.
Hence, $\operatorname{cone}(f_{<})$, the mapping cone of $f_{<}$, is path connected.
As the bundle morphism $F_{<} \colon \operatorname{ft}_{<} E \rightarrow E$ restricts on fibers to copies of the map $f_{<}$, it follows that $\operatorname{cone}(F_{<})$ is path connected.

It remains to show that $\pi_{1}(\operatorname{cone}(F_{<}), \operatorname{pt})$ is trivial.
For this purpose, we note that the bundle morphism $F_{<} \colon \operatorname{ft}_{<} E \rightarrow E$ can be written as the disjoint union $F_{<} = \bigsqcup_{i} F_{<}^{(i)}$ of bundle morphisms
$$
F_{<}^{(i)} \colon \operatorname{ft}_{<} E^{(i)} \rightarrow E^{(i)},
$$
where $\operatorname{ft}_{<} E = \bigsqcup_{i} \operatorname{ft}_{<} E^{(i)}$ and $E = \bigsqcup_{i} E^{(i)}$ are decompositions into the path components, and we note that the bundles $\pi_{<} \colon \operatorname{ft}_{<} E \rightarrow B$ and $\pi \colon E \rightarrow B$ restrict to bundles $\pi_{<}^{(i)} \colon \operatorname{ft}_{<} E^{(i)} \rightarrow B$ and $\pi^{(i)} \colon E^{(i)} \rightarrow B$.
(In fact, from every point $x \in E$ there exists a path in $E$ to a point of the form $F_{<}(y)$ for some $y \in \operatorname{ft}_{<} E$ because $F_{<}$ restricts over the basepoint $\pi(x) \in B$ to a copy of the map $f_{<} \colon L_{<} \rightarrow L$, and we can use surjectivity of the induced map $H_{0}(L_{<}) \rightarrow H_{0}(L)$.
Conversely, given two points $y_{0}, y_{1} \in \operatorname{ft}_{<} E$ whose images under $F_{<}$ can be connected by a path in $E$, we construct a path between $y_{0}$ and $y_{0}$ in $\operatorname{ft}_{<} E$ as follows.
We choose a path $\gamma_{E} \colon [0, 1] \rightarrow E$ from $\gamma_{E}(0) = F_{<}(y_{0})$ to $\gamma_{E}(1) = F_{<}(y_{1})$.
Then, $\gamma_{E}$ is a section of $\pi \colon E \rightarrow B$ along the path $\gamma = \pi \circ \gamma_{E} \colon [0, 1] \rightarrow B$. In particular we have $F_{<}(y_{0}) \in \pi^{-1}(\gamma(0))$ and $F_{<}(y_{1}) \in \pi^{-1}(\gamma(1))$.
The pullback of the bundle morphism $F_{<} \colon \operatorname{ft}_{<} E \rightarrow E$ under $\gamma$ is isomorphic to the bundle morphism $f_{<} \times \operatorname{id}_{[0, 1]} \colon L_{<} \times [0, 1] \rightarrow L \times [0, 1]$ of trivial bundles over the interval $[0, 1]$.
Under this pullback, the points $y_{0} \in \pi_{<}^{-1}(\gamma(0))$ and $y_{1} \in \pi_{<}^{-1}(\gamma(1))$ correspond to points of the form $(y_{0}', 0) \in L_{<} \times \{0\}$ and $(y_{1}', 1) \in L_{<} \times \{1\}$, respectively.
Furthermore, the section $\gamma_{E} \colon [0, 1] \rightarrow E$ of $\pi \colon E \rightarrow B$ pulls back to a section $\gamma_{E}' \colon [0, 1] \rightarrow L \times [0, 1]$ of the projection $L \times [0, 1] \rightarrow [0, 1]$ such that $\gamma_{E}'(0) = (f_{<}(y_{0}'), 0) \in L \times \{0\}$ and $\gamma_{E}'(1) = (f_{<}(y_{1}'), 1) \in L \times \{1\}$.
Using injectivity of the map $H_{0}(L_{<}) \rightarrow H_{0}(L)$ induced by $f_{<}$, we conclude that there is a section $\gamma_{<}' \colon [0, 1] \rightarrow L_{<} \times [0, 1]$ of the projection $L_{<} \times [0, 1] \rightarrow [0, 1]$ such that $\gamma_{<}'(0) = (y_{0}', 0)$ and $\gamma_{<}'(1) = (y_{1}', 1)$.
All in all, the composition of $\gamma_{<}' \colon [0, 1] \rightarrow L_{<} \times [0, 1]$ with the structure map $L_{<} \times [0, 1] \rightarrow \operatorname{ft}_{<} E$ is the desired path in $\operatorname{ft}_{<} E$.)
Since $\operatorname{cone}(F_{<}) = \bigvee_{i} \operatorname{cone}(F_{<}^{(i)})$, we conclude from the Seifert-van Kampen theorem for a bouquet of spaces (see Example 1.21 in \cite[p. 43]{hat}) that in order to show that $\pi_{1}(\operatorname{cone}(F_{<}), \operatorname{pt})$ is trivial, it suffices to show that $\pi_{1}(\operatorname{cone}(F_{<}^{(i)}), \operatorname{pt})$ is trivial for every $i$.

Fix an index $i$.
The Seifert-van Kampen theorem implies that the fundamental group $\pi_{1}(\operatorname{cone}(F_{<}^{(i)}), \operatorname{pt})$ is trivial when the canonical inclusion $\operatorname{ft}_{<}E^{(i)} \subset \operatorname{cyl}(F_{<}^{(i)})$ induces a surjection on fundamental groups.
In view of the exactness of the sequence
$$
\pi_{1}(\operatorname{ft}_{<}E^{(i)}) \rightarrow \pi_{1}(\operatorname{cyl}(F_{<}^{(i)})) \rightarrow \pi_{1}(\operatorname{cyl}(F_{<}^{(i)}), \operatorname{ft}_{<}E^{(i)}),
$$
it suffices to show that $\pi_{1}(\operatorname{cyl}(F_{<}^{(i)}), \operatorname{ft}_{<}E^{(i)})$ is trivial.
By the discussion in \cite[p. 346]{hat}, this means we have to show that every path $[0, 1] \rightarrow \operatorname{cyl}(F_{<}^{(i)})$ with endpoints in $\operatorname{ft}_{<}E^{(i)}$ is homotopic rel endpoints to a path $[0, 1] \rightarrow \operatorname{ft}_{<}E^{(i)}$.

Let us show that every path $\alpha \colon [0, 1] \rightarrow \operatorname{cyl}(F_{<})$ with endpoints in $\operatorname{ft}_{<}E \subset \operatorname{cyl}(F_{<})$ is homotopic rel endpoints to a path $[0, 1] \rightarrow \operatorname{ft}_{<}E$.
Note that we have a fiber bundle $\rho \colon \operatorname{cyl}(F_{<}) \rightarrow B$ with fiber $\operatorname{cyl}(f_{<})$, where the canonical inclusion $\operatorname{ft}_{<}E \subset \operatorname{cyl}(F_{<})$ can fiberwisely be identified with the canonical inclusion $j \colon L_{<} \rightarrow \operatorname{cyl}(f_{<})$.
Then, we may consider $\alpha$ as a section of $\rho \colon \operatorname{cyl}(F_{<}) \rightarrow B$ along the path $\beta = \rho \circ \alpha \colon [0, 1] \rightarrow B$.
The pullback of the bundle $\rho \colon \operatorname{cyl}(F_{<}) \rightarrow B$ under $\beta$ can be identified with the trivial bundle $\rho' = \operatorname{pr}_{[0, 1]} \colon \operatorname{cyl}(f_{<}) \times [0, 1] \rightarrow [0, 1]$ in such a way that the section $\alpha$ of $\rho$ pulls back to a section $\alpha' \colon [0, 1] \rightarrow \operatorname{cyl}(f_{<}) \times [0, 1]$ of $\rho'$, where $\alpha'(0) \in L_{<} \times \{0\}$ and $\alpha'(1) \in L_{<} \times \{1\}$.
It remains to show that the unique map $\alpha'' \colon [0, 1] \rightarrow \operatorname{cyl}(f_{<})$ satisfying $\alpha'(t) = (\alpha''(t), t)$ for all $t \in [0, 1]$ is homotopic rel endpoints to a map $[0, 1] \rightarrow L_{<}$.
In other words, by the discussion in \cite[p. 346]{hat}, we have to show that $\pi_{1}(\operatorname{cyl}(f_{<}), L_{<}, x_{0}) = 0$ for every basepoint $x_{0} \in L_{<}$.
In order to show this, we consider the following portion of the long exact sequence of homotopy groups of the pointed pair $(\operatorname{cyl}(f_{<}), L_{<}, x_{0})$ (omitting the basepoint):
$$
\pi_{1}(L_{<}) \stackrel{j_{\ast}}{\longrightarrow} \pi_{1}(\operatorname{cyl}(f_{<})) \rightarrow \pi_{1}(\operatorname{cyl}(f_{<}), L_{<}) \rightarrow \pi_{0}(L_{<}) \stackrel{j_{\ast}}{\longrightarrow} \pi_{0}(\operatorname{cyl}(f_{<})).
$$
Note that $j_{\ast}$ is an isomorphism in degree $0$ because $f_{<}$ induces an isomorphism $\pi_{0}(L_{<}, x) \rightarrow \pi_{0}(L, x)$.
Thus, the claim follows from the assumption that $f_{<}$ induces a surjection on fundamental groups, i.e., that $j_{\ast}$ is a surjection in degree $1$.
\end{proof}

\begin{example}\label{remark abelian fundamental group}
If $L$ is path connected and has abelian fundamental group, then any equivariant Moore approximation $f_{<} \colon L_{<} \rightarrow L$ to $L$ of degree $\geq 2$ induces automatically a surjection on fundamental groups.
In fact, by naturality of the Hurewicz homomorphism we have for any basepoint $x_{0} \in L_{<}$ the commutative diagram
\begin{center}
\begin{tikzpicture}

\path (0,0) node(a) {$\pi_{1}\left(L_{<}, x_{0}\right)$}
         (4, 0) node(b) {$\pi_{1}\left(L, f_{<}(x_{0})\right)$}
         (0, -2) node(c) {$H_{1}\left(L_{<}\right)$}
         (4, -2) node(d) {$H_{1}\left(L\right)$.};

\draw[->] (a) -- node[left] {$\operatorname{Hur}$} (c);
\draw[->] (a) -- node[above] {$f_{< \ast}$} (b);
\draw[->] (b) -- node[right] {$\operatorname{Hur}$} (d);
\draw[->] (c) -- node[above] {$f_{< \ast}$} (d);
\end{tikzpicture}
\end{center}
Note that the left vertical map is surjective and the right vertical map is an isomorphism because in degree one, Hurewicz maps are abelianization maps.
Hence, the claim follows because $f_{< \ast} \colon H_{1}\left(L_{<}\right) \rightarrow H_{1}(L)$ is an isomorphism.
\end{example}

\begin{example}\label{remark trivial link bundle}
If $L$ is a path connected cell complex and the fiber bundle $E \rightarrow B$ is trivial, then there exists for any $d \geq 2$ an equivariant Moore approximation $f_{<} \colon L_{<} \rightarrow L$ to $L$ of degree $d$ which induces a surjection on fundamental groups.
Indeed, since the bundle $E \rightarrow B$ is trivial, the structure group $G$ can be chosen to be trivial.
Thus, by the results of \cite{wra}, $L$ admits a cellular Moore approximation $f_{<} \colon L_{<} \rightarrow L$ of degree $d \geq 2$ in such a way that $f_{<}$ restricts to the identity on $(d-1)$-skeleta.
In particular, the map induced by $f_{<}$ on fundamental groups is surjective.
\end{example}

\section{Rational homotopy theory}\label{rational homotopy theory}
\Cref{main theorem} reveals a natural relation between the rational Hurewicz homomorphism of the truncation cone and the condition that in some degree all complementary cup products in the reduced cohomology ring of the truncation cone vanish.
In this section we study the homotopy theoretic ramifications of this cohomological vanishing condition (see e.g. \Cref{corollary hurewicz theorem}).

The main purpose of this section is prove (see \Cref{proof of ration poincare duality space}) the following

\begin{theorem}\label{proposition rational poincare duality space}
Let $n \geq 3$ be an integer.
Let $\pi \colon E^{n-1} \rightarrow B$ be a fiber bundle of connected closed manifolds with connected closed manifold fiber $L$ and structure group $G$ such that $B$, $E$ and $L$ are compatibly oriented.
Suppose that $B$ admits a good open cover (this holds whenever $B$ is smooth or at least PL).
Let $\overline{p}$ and $\overline{q}$ be complementary perversities, and set $k = c-\overline{p}(c+1)$ and $l = c-\overline{q}(c+1)$, where $c = \operatorname{dim} L$.
Suppose that the fiber $L$ possesses a $G$-equivariant Moore approximation $f_{<} \colon L_{<} \rightarrow L$ both of degree $k$ and of degree $l$ such that $f_{<}$ induces a surjection on fundamental groups.
Furthermore, we suppose that
\begin{enumerate}[(1)]
\item $n \leq 3 \cdot \operatorname{max}\{k, l\} -1$, and
\item $\mathcal{O}_{i}(\pi, k, l) = 0$ for all $i$.
\end{enumerate}
Then, the rational Hurewicz homomorphism of the associated truncation cone,
$$
\operatorname{Hur}_{n-1 \ast} \colon \pi_{n-1}(\operatorname{cone}(F_{<}), \operatorname{pt}) \otimes_{\mathbb{Z}} \mathbb{Q} \rightarrow H_{n-1}(\operatorname{cone}(F_{<})),
$$
is surjective.
\end{theorem}

\subsection{Minimal Sullivan algebras}\label{minimal sullivan algebras}
Assuming the ground field to be the rationals, we provide some necessary notation from \cite{fht}.

By a \emph{graded (rational) vector space} we mean a collection $V = \{V^{p}\}_{p \geq 1}$ of rational vector spaces (see \cite[p. 40]{fht} and \cite[p. 138]{fht}).
We say that $v \in V^{p}$ is an element of $V$ of degree $p$, and write $v \in V$ and $|v| = p$.
As in \cite[p. 42]{fht}, we use the notation $V^{\geq p} = \{V^{q}\}_{q \geq p}$, $V^{< p} = \{V^{q}\}_{q < p}$, etc.

As in Example 6 of \cite[p. 45]{fht}, the \emph{free graded commutative algebra} of $V$ is the quotient algebra
$\Lambda V = TV/I$,
where $TV$ denotes the tensor algebra of $V$ (see Example 4 in \cite[p. 45]{fht}), and $I \subset TV$ denotes the ideal generated by all elements of the form $v \otimes w - (-1)^{|v| \cdot |w|} w \otimes v$, where $v, w \in V$.
Following \cite[p. 140]{fht}, we also write $\Lambda V^{\geq p} = \Lambda (V^{\geq p})$, $\Lambda V^{< p} = \Lambda (V^{< p})$, etc.

Note that, as vector spaces, $\Lambda V = \bigoplus_{p=0}^{\infty}\Lambda^{p} V$, where $\Lambda^{p} V$ denotes the linear span of all products $v_{1} \wedge \dots \wedge v_{p}$, $v_{1}, \dots, v_{p} \in V$, of word length $p$ and degree $|v_{1}| + \dots + |v_{p}|$ (see Example 6(vi) in \cite[p. 46]{fht} as well as \cite[p. 140]{fht}).
We will also write $\Lambda^{\geq p} V = \bigoplus_{q \geq p} \Lambda^{q} V$, and in particular $\Lambda^{+} V = \Lambda^{\geq 1} V$.

Recall from \cite[p. 138]{fht}) that a \emph{Sullivan algebra} is a commutative cochain algebra of the form $(\Lambda V, d)$ whose differential $d$ satisfies the following nilpotence condition.
It is required that there exists an increasing sequence $V(0) \subset \dots \subset V(k) \subset \dots$ of graded linear subspaces of $V$ such that $V = \bigcup_{k=0}^{\infty}V(k)$, $d=0$ on $V(0)$, and $d(V(k)) \subset \Lambda V(k-1)$ for all $k > 0$.
If the differential $d$ of the Sullivan algebra $(\Lambda V, d)$ satisfies in addition $\operatorname{im} d \subset \Lambda^{\geq 2} V$, then we call the Sullivan algebra $(\Lambda V, d)$ \emph{minimal}.

Note that if a Sullivan algebra $(\Lambda V, d)$ is minimal, then the projection
$$
\rho \colon \Lambda^{+} V = \bigoplus_{q \geq 1} \Lambda^{q} V \rightarrow \Lambda^{1} V = V
$$
induces according to the discussion in \cite[p. 173]{fht} a homomorphism
$$
\zeta \colon H^{+}(\Lambda V) = \frac{\operatorname{ker}(d \colon \Lambda^{+} V \rightarrow \Lambda^{+} V)}{\operatorname{im}(d \colon \Lambda^{+} V \rightarrow \Lambda^{+} V)} \rightarrow V, \qquad \zeta([z]) = \rho(z).
$$

\begin{lemma}\label{lemma non-vanishing zeta homomorphism}
Fix integers $r, s, t \geq 1$.
Suppose that $(\Lambda V, d)$ is a minimal Sullivan algebra whose underlying graded vector space $V$ is of the form $V = \{V^{p}\}_{p \geq r}$, and whose differential $d$ vanishes on all elements of degree $\leq s$.
Furthermore, suppose that all $t$-complementary products in $H^{+}(\Lambda V)$ vanish, that is, for all classes $\alpha \in H^{i}(\Lambda V)$, $\beta \in H^{t-i}(\Lambda V)$, $1 \leq i \leq t-1$, we have $\alpha \wedge \beta = 0$ in $H^{t}(\Lambda V)$.
If $t \leq r+s$, then the associated homomorphism $\zeta \colon H^{+}(\Lambda V) \rightarrow V$ is injective in degree $t$.
\end{lemma}

\begin{proof}
Suppose that $[z] \in H^{t}(\Lambda V)$ is a class such that $\zeta([z]) = 0$.
In particular, $z$ is a cocycle of degree $t$ in $(\Lambda V, d)$ satisfying $z \in \operatorname{im} d$.
Thus, since $\operatorname{im} d \subset \Lambda^{\geq 2} V$, $z$ can be written as a rational linear combination of elements of the form $v_{1} \wedge \dots \wedge v_{m}$, where $m \geq 2$, and $v_{1}, \dots, v_{m} \in V$ are nonzero elements such that $|v_{1}| + \dots + |v_{m}| = t$.
Since $m \geq 2$, it follows from $|v_{1}|, \dots, |v_{m}| \geq r$ and $t \leq r+s$ that $|v_{1}|, \dots, |v_{m}| \leq s$.
Consequently, the elements $v_{1}, \dots, v_{m} \in V^{\leq s}$ are all cocycles in $(\Lambda V, d)$.
Hence, $[v_{1}] \wedge \dots \wedge [v_{m}] = 0$ in $H^{t}(\Lambda V)$ according to the vanishing assumption on $t$-complementary products in $H^{+}(\Lambda V)$.
This shows that $[z] = 0$ in $H^{t}(\Lambda V)$.
\end{proof}

The following example shows that the assumption $t \leq r+s$ is in general necessary in \Cref{lemma non-vanishing zeta homomorphism}.

\begin{example}\label{example minimal model}
Fix an odd integer $u \geq 3$.
Suppose that the only nonzero parts of the graded vector space $V$ are $V^{u}$ with basis $\{x, y\}$, and $V^{2u-1}$ with basis $\{z\}$.
Then, it is easy to check that $(\Lambda V, d)$ is a minimal Sullivan algebra with differential $d$ determined by $d x = d y = 0$ and $d z = xy$.
By construction, the only non-trivial cohomology groups of $H^{+}(\Lambda V)$ are $H^{u}(\Lambda V)$ (generated by the classes of $x$ and $y$), $H^{3u-1}(\Lambda V)$ (generated by the classes of $xz$ and $yz$), and $H^{4u-1}(\Lambda V)$ (generated by the class of $xyz$).
Now let $r = u$, $s = 2u-2$, and $t = r+s+1 = 3u-1$.
Then, the minimal Sullivan algebra $(\Lambda V, d)$ satisfies all assumptions of \Cref{lemma non-vanishing zeta homomorphism} except for $t \leq r+s$.
Furthermore, the homomorphism $\zeta \colon H^{+}(\Lambda V) \rightarrow V$ is clearly not injective in degree $t$ because $V^{3u-1} = 0$, whereas $H^{3u-1}(\Lambda V) \neq 0$.
\end{example}

\subsection{Minimal Sullivan models}
Given any commutative cochain algebra $(A, d)$, a \emph{Sullivan model} for $(A, d)$ is a Sullivan algebra $(\Lambda V, d)$ together with a quasi-isomorphism $m \colon (\Lambda V, d) \stackrel{\simeq}{\longrightarrow} (A, d)$.
The Sullivan model is called \emph{minimal} if the corresponding Sullivan algebra $(\Lambda V, d)$ is minimal.
It can be shown that any commutative cochain algebra $(A, d)$ with $H^{0}(A) \cong \mathbb{Q}$ possesses a unique minimal Sullivan model (see the corollary to Theorem 14.12 in \cite[p. 191]{fht}).

Given a commutative cochain algebra $(A, d)$ such that $H^{0}(A) \cong \mathbb{Q}$ and $H^{1}(A) = 0$, the following explicit algorithm for constructing the (unique) minimal Sullivan model $m \colon (\Lambda V, d) \stackrel{\simeq}{\longrightarrow} (A, d)$ is described before Proposition 12.2 in \cite[p. 145]{fht}.
Starting with a morphism $m_{2} \colon (\Lambda V^{2}, 0) \rightarrow (A, d)$ such that $H^{2}(m_{2}) \colon V^{2} \stackrel{\cong}{\longrightarrow} H^{2}(A)$, the construction provides inductively for $k = 2, 3, \dots$ a vector space $V^{k+1}$, an extension of the (derivational) differential $d$ from $\Lambda V^{\leq k}$ to $\Lambda V^{\leq k+1} = \Lambda V^{\leq k} \otimes \Lambda V^{k+1}$, and an extension of $m_{k} \colon (\Lambda V^{\leq k}, d) \rightarrow (A, d)$ to a cochain algebra morphism $m_{k+1} \colon (\Lambda V^{\leq k+1}, d) \rightarrow (A, d)$, such that, for all $k = 2, 3, \dots$, $H^{i}(m_{k})$ is an isomorphism for $i \leq k$, and $H^{k+1}(m_{k})$ is injective.
Having constructed $m_{k} \colon (\Lambda V^{\leq k}, d) \rightarrow (A, d)$, the next step of the induction is as follows.
Choose cocycles $a_{\alpha} \in A^{k+1}$ and $z_{\beta} \in (\Lambda V^{\leq k})^{k+2}$ such that
\begin{align*}
H^{k+1}(A) &= \operatorname{im} H^{k+1}(m_{k}) \oplus \bigoplus_{\alpha} \mathbb{Q} \cdot [a_{\alpha}], \\
\operatorname{ker} H^{k+2}(m_{k}) &= \bigoplus_{\beta} \mathbb{Q} \cdot [z_{\beta}].
\end{align*}
Let $V^{k+1}$ be a vector space with basis $\{v'_{\alpha}\} \cup \{v''_{\beta}\}$ in correspondence with the elements $\{a_{\alpha}\} \cup \{z_{\beta}\}$.
The (derivational) differential $d$ is extended from $\Lambda V^{\leq k}$ to $\Lambda V^{\leq k+1} = \Lambda V^{\leq k} \otimes \Lambda V^{k+1}$ by setting $d v'_{\alpha} = 0$ and $d v''_{\beta} = z_{\beta}$.
Finally, the map $m_{k} \colon (\Lambda V^{\leq k}, d) \rightarrow (A, d)$ is extended to a cochain algebra morphism
$$
m_{k+1} \colon (\Lambda V^{\leq k+1}, d) \rightarrow (A, d)
$$
via $m_{k+1} v'_{\alpha} = a_{\alpha}$ and $m_{k+1} v''_{\beta} = b_{\beta}$, where $b_{\beta} \in A^{k+1}$ has been chosen in such a way that $d b_{\beta} = m_{k} z_{\beta}$.

\begin{lemma}\label{lemma minimal sullivan model}
Consider a commutative cochain algebra $(A, d)$ satisfying $H^{0}(A) \cong \mathbb{Q}$ and $H^{1}(A) = 0$.
Let $r \geq 2$ be the smallest integer such that $H^{r}(A) \neq 0$.
If $m \colon (\Lambda V, d) \stackrel{\simeq}{\longrightarrow} (A, d)$ denotes the (unique) minimal Sullivan model for $(A, d)$, then the underlying graded vector space $V$ is of the form $V = \{V^{p}\}_{p \geq r}$, and the differential $d$ vanishes on all elements of degree $\leq 2r-2$.
\end{lemma}

\begin{proof}
The claim $V = \{V^{p}\}_{p \geq r}$ is part of Proposition 12.2(ii) in \cite[p. 145]{fht}.

In order to show that the differential $d$ vanishes on all elements of degree $\leq 2r-2$, it suffices by the inductive construction of $d$ to show that $\operatorname{ker} H^{k+2}(m_{k}) = 0$ for $k= 2, \dots, 2r-3$.
(For $r = 2$ there is nothing to show since $d$ vanishes on all elements of degree $\leq 2$ by construction.)
Fix $k= 2, \dots, 2r-3$.
If $k < r$, then $V = \{V^{p}\}_{p \geq r}$ implies that $(\Lambda V^{\leq k})^{0} = \mathbb{Q}$ and $(\Lambda V^{\leq k})^{p} = 0$ for $p > 0$.
If, however, $k \geq r$, then $V = \{V^{p}\}_{p \geq r}$ implies that
\begin{align*}
(\Lambda V^{\leq k})^{p} = \begin{cases}
\mathbb{Q}, \quad &p = 0, \\
0, \quad &p = 1, \dots, r-1, \\
V^{p}, \quad &p = r, \dots, k, \\
0, \quad &p = k+1, \dots, 2r-1.
\end{cases}
\end{align*}
In any case, we see that $(\Lambda V^{\leq k})^{k+2} = 0$, and, in particular, $\operatorname{ker} H^{k+2}(m_{k}) = 0$.
\end{proof}

\subsection{Commutative cochain algebras for spaces}
Recall from Section 10 in \cite[p. 115 ff.]{fht} that Sullivan has constructed a contravariant functor $A_{PL}$ from the category of topological spaces and continuous maps to the category of commutative cochain algebras and cochain algebra morphisms.
Furthermore, the functor $A_{PL}$ has the important property that for any topological space $X$, the graded algebras $H^{\ast}(X)$ and $H(A_{PL}(X))$ are naturally isomorphic (and can hence be identified).

Given a path connected topological space $X$, we can in particular consider the minimal Sullivan model for $X$, that is, the (unique) minimal Sullivan model $m_{X} \colon (\Lambda V_{X}, d) \stackrel{\simeq}{\rightarrow} A_{PL}(X)$ for the commutative cochain algebra $(A, d) = A_{PL}(X)$.
Moreover, recall from \Cref{minimal sullivan algebras} that we can associate to the minimal Sullivan algebra $(\Lambda V_{X}, d)$ a homomorphism $\zeta_{X} \colon H^{+}(\Lambda V_{X}) \rightarrow V_{X}$.

In \Cref{proposition hurewicz and sullivan} below, we characterize the non-vanishing of the rational Hurewicz homomorphism in terms of rational homotopy theory (compare the proof of Proposition 3.14 in \cite[p. 250]{klim}).

\begin{proposition}\label{proposition hurewicz and sullivan}
Let $(X, x_{0})$ be a simply connected pointed space, and suppose that $H_{\ast}(X)$ is of finite type (i.e., the rational vector space $H_{r}(X)$ has finite dimension for all $r \in \mathbb{Z}$).
Then for any integer $n \geq 3$ the following statements are equivalent:
\begin{enumerate}[$(i)$]
\item The Hurewicz homomorphism
$$
\operatorname{Hur}_{n-1 \ast} \colon \pi_{n-1}(X, x_{0}) \otimes_{\mathbb{Z}} \mathbb{Q} \rightarrow H_{n-1}(X)
$$
is surjective.
\item The homomorphism $\zeta_{X}^{n-1} \colon H^{n-1}(\Lambda V_{X}) \rightarrow V_{X}^{n-1}$ is injective.
\end{enumerate}
Moreover, if either of the above statements is satisfied for $X$, then all $(n-1)$-complementary cup products in $\widetilde{H}^{\ast}(X)$ vanish.
\end{proposition}

\begin{proof}
We fix a minimal Sullivan model $m_{X} \colon (\Lambda V_{X}, d) \stackrel{\simeq}{\longrightarrow} A_{PL}(X)$ for $X$.

Since $(X, x_{0})$ is simply connected and $H_{\ast}(X)$ is of finite type, the corollary to Theorem 15.11 in \cite[p. 210]{fht} implies that in degree $n-1 > 1$ the homomorphism
$\zeta_{X}^{n-1} \colon H^{n-1}(\Lambda V_{X}) \rightarrow V_{X}^{n-1}$
can be identified with the dual
$$
\operatorname{Hur}_{n-1}^{\ast} \colon H^{n-1}(X) \rightarrow \operatorname{Hom}_{\mathbb{Z}}(\pi_{n-1}(X), \mathbb{Q})
$$
of the Hurewicz map $\operatorname{Hur}_{n-1} \colon \pi_{n-1}(X) \rightarrow H_{n-1}(X; \mathbb{Z})$. 
Since the extension of scalars functor $- \otimes_{\mathbb{Z}} \mathbb{Q}$ is left adjoint to the restriction of scalars functor (along $\mathbb{Z} \subset \mathbb{Q}$), the latter homomorphism can furthermore be identified with
$$
\operatorname{Hom}_{\mathbb{Q}}(\operatorname{Hur}_{n-1 \ast}(-), \mathbb{Q}) \colon \operatorname{Hom}_{\mathbb{Q}}(H_{n-1}(X) \otimes_{\mathbb{Z}} \mathbb{Q}, \mathbb{Q}) \rightarrow \operatorname{Hom}_{\mathbb{Q}}(\pi_{n-1}(X) \otimes_{\mathbb{Z}} \mathbb{Q}, \mathbb{Q}).
$$
Thus, the equivalence $(i) \Leftrightarrow (ii)$ follows from the observation that a homomorphism $V \rightarrow W$ of rational vector spaces is surjective if and only if its dual $\operatorname{Hom}_{\mathbb{Q}}(W, \mathbb{Q}) \rightarrow \operatorname{Hom}_{\mathbb{Q}}(V, \mathbb{Q})$ is injective.

Let us suppose that there is a non-vanishing $(n-1)$-complementary cup product in $\widetilde{H}^{\ast}(X)$, that is, for some $i \in \{1, \dots, n-2\}$, there exist classes $\alpha \in H^{i}(X)$ and $\beta \in H^{n-1-i}(X)$ such that $0 \neq \alpha \cup \beta \in H^{n-1}(X)$.
As $H^{\ast}(A_{PL}(X)) \cong H^{\ast}(X)$ (as graded algebras) and $m_{X} \colon (\Lambda V_{X}, d) \stackrel{\simeq}{\longrightarrow} A_{PL}(X)$ is a quasi-isomorphism of commutative differential graded algebras, there exist cocycles $a \in \Lambda V_{X}$ of degree $i$ and $b \in \Lambda V_{X}$ of degree $n-1-i$ such that $0 \neq [a \wedge b] \in H^{n-1}( \Lambda V_{X})$.
But then $\zeta_{X}^{n-1}$ is not injective because $a \wedge b \in \Lambda^{\geq 2} V$ implies that $\zeta_{X}([a \wedge b]) = \rho(a \wedge b) = 0$.
Hence, we have shown that all $(n-1)$-complementary cup products in $\widetilde{H}^{\ast}(X)$ must vanish if $X$ satisfies either of the statements $(i)$ and $(ii)$.
\end{proof}
\begin{remark}
In the context of \Cref{main theorem} we have exploited statement $(i)$ in our proof that the local duality obstructions vanish.
Under the assumption that $X  = \operatorname{cone}(F_{<})$ is simply connected (compare \Cref{lemma simply connected cone}), \Cref{proposition hurewicz and sullivan} provides another proof by means of the minimal Sullivan model of $X$, based on statement $(ii)$.
\end{remark}

\subsection{A rational Hurewicz theorem}
One part of \Cref{proposition hurewicz and sullivan} states that in a given degree, surjectivity of the rational Hurewicz homomorphism implies that complementary cup products in the reduced cohomology ring vanish.
Our purpose is to prove the converse implication under reasonable hypotheses.
This is addressed in \Cref{corollary hurewicz theorem}, where the additional assumption is that the considered degree $n-1$ should not be too large compared to the rational connectedness $r$ of the space $X$.
Note that for $n-1 = 2r-1$, \Cref{corollary hurewicz theorem} specializes to a part of the classical rational Hurewicz theorem (for spaces with homology of finite type) because the vanishing condition for complementary cup products is trivially satisfied.
This is exactly the implication that is employed in \cite{klim} in the context of isolated singularities.

\begin{corollary}\label{corollary hurewicz theorem}
Let $(X, x_{0})$ be a simply connected pointed space, and suppose that $H_{\ast}(X)$ is of finite type (i.e., the rational vector space $H_{r}(X)$ has finite dimension for all $r \in \mathbb{Z}$).
Let $n \geq 3$ be an integer such that the following assumptions hold:
\begin{enumerate}[(1)]
\item If $r \geq 2$ denotes the smallest positive integer such that $H^{r}(X) \neq 0$, then $n \leq 3r-1$.
\item All $(n-1)$-complementary cup products in $\widetilde{H}^{\ast}(X)$ vanish, that is, for all $\alpha \in H^{i}(X)$, $\beta \in H^{n-1-i}(X)$, $1 \leq i \leq n-2$, we have $0 = \alpha \cup \beta \in H^{n-1}(X)$.
\end{enumerate}
Then, the rational Hurewicz homomorphism
$$
\operatorname{Hur}_{n-1 \ast} \colon \pi_{n-1}(X, x_{0}) \otimes_{\mathbb{Z}} \mathbb{Q} \rightarrow H_{n-1}(X)
$$
is surjective.
\end{corollary}

\begin{proof}
We consider the commutative cochain algebra $(A, d) = A_{PL}(X)$ and its minimal Sullivan model $m_{X} \colon (\Lambda V_{X}, d) \stackrel{\simeq}{\longrightarrow} A_{PL}(X)$.
Using that the graded algebras $H^{\ast}(X)$ and $H^{\ast}(A_{PL}(X))$ are isomorphic, we obtain from \Cref{lemma minimal sullivan model} that $V_{X} = \{V_{X}^{p}\}_{p \geq r}$, and that the differential $d$ vanishes on all elements of degree $\leq 2r-2$, where $r \geq 2$ denotes the smallest positive integer such that $H^{r}(X) \neq 0$.
Next, we apply \Cref{lemma non-vanishing zeta homomorphism} to the minimal Sullivan algebra $(\Lambda V_{X}, d)$ and the integers $r$, $s = 2r-2$ and $t = n-1$.
Note that $t \leq r + s$ by assumption (1), and all $t$-complementary cup products in $H^{+}(\Lambda V_{X})$ vanish by assumption (2) because $m_{X}$ is a quasi-isomorphism.
Hence, we conclude that the homomorphism $\zeta_{X} \colon H^{+}(\Lambda V_{X}) \rightarrow V_{X}$ is injective in degree $t = n-1$.
Finally, the claim that the Hurewicz homomorphism $\operatorname{Hur}_{n-1 \ast} \colon \pi_{n-1}(X, x_{0}) \otimes_{\mathbb{Z}} \mathbb{Q} \rightarrow H_{n-1}(X)$ is surjective follows from implication $(ii) \Rightarrow (i)$ of \Cref{proposition hurewicz and sullivan}.
\end{proof}

\subsection{Proof of \Cref{proposition rational poincare duality space}}\label{proof of ration poincare duality space}
Since the manifolds $B$ and $L$ are path connected and the $G$-equivariant Moore approximation $f_{<} \colon L_{<} \rightarrow L$ induces by assumption a surjection on fundamental groups, we conclude from \Cref{lemma simply connected cone} that the truncation cone $\operatorname{cone}(F_{<})$ is simply connected.
Next, we wish to apply \Cref{corollary hurewicz theorem} to the simply connected pointed space $(X, x_{0}) = (\operatorname{cone}(F_{<}), \operatorname{pt})$. Note that $H_{\ast}(X)$ is of finite type because the inclusion $E \subset \operatorname{cone}(F_{<})$ induces a surjective homomorphism $H_{r}(E) \rightarrow H_{r}(\operatorname{cone}(F_{<}))$ for all $r \in \mathbb{Z}$ (see properties (a) and (b) of $Q_{\geq}E$ listed in the proof of \Cref{main theorem} in \Cref{proof of main theorem}, where we use the assumption that $B$ admits a good open cover).
The assumptions (1) and (2) of \Cref{proposition rational poincare duality space} imply that the conditions (1) and (2) of \Cref{corollary hurewicz theorem} are satisfied.
(Concerning (1), note that the smallest integer $r \geq 2$ such that $H^{r}(X) \neq 0$ satisfies $r \geq \operatorname{max}\{k, l\}$ according to \Cref{lemma homology of truncation cone}, where we use again the assumption that $B$ admits a good open cover.)

Finally, application of \Cref{corollary hurewicz theorem} completes the proof of \Cref{proposition rational poincare duality space}.

\section{Intersection spaces and the signature}\label{intersection spaces and the signature}
In \Cref{Witt spaces and signature} we apply the results of the previous sections to study middle perversity intersection spaces of depth one Witt spaces within the framework of Thom-Mather stratified spaces.
In particular, given a depth one Witt space $X$ (see \Cref{definition witt space}), \Cref{main corollary}(a) provides our Hurewicz criterion for the existence of a Klimczak completion $\widehat{IX} = IX \cup e^{n}$ of the middle perversity intersection space $IX$.
Then, in part (b) of \Cref{main corollary}, we relate the Hurewicz criterion to the vanishing of the local duality obstructions to Banagl-Chriestenson by showing that they are equivalent if the dimensions of the singular strata are not too big.
Moreover, when the dimension $n$ of $X$ is of the form $n = 4d$, \Cref{main corollary}(c) implies that the signature of the symmetric intersection form $H_{2d}(\widehat{IX}) \times H_{2d}(\widehat{IX}) \rightarrow \mathbb{Q}$ equals the signature of the Goresky-MacPherson-Siegel intersection form $IH_{2d}(X) \times IH_{2d}(X) \rightarrow \mathbb{Q}$ on middle-perversity intersection homology.
Finally, we illustrate our results by a concrete example in \Cref{An example}.

\subsection{Depth one Witt spaces}\label{Witt spaces and signature}
Before focusing on depth one Witt spaces and their middle perversity intersection spaces (see \Cref{main corollary}), we discuss more generally intersection spaces of stratified pseudomanifold of depth $1$ along the lines of \cite{bc}.

First of all, an $n$-dimensional two strata pseudomanifold is a pair $(X, \Sigma)$ consisting of a locally compact, second countable Hausdorff space $X$ and a closed connected subspace $\Sigma$ such that $\Sigma \subset X$ is equipped with a Thom-Mather $C^{\infty}$-stratification of $X$ as follows.
The regular stratum $X \setminus \Sigma$ is a smooth $n$-manifold that is dense in $X$, and the singular stratum $\Sigma$ is a smooth manifold whose codimension in $X$ is at least $2$.
Moreover, $\Sigma$ is required to be equipped with so-called Thom-Mather control data (for details, see Section 8 of \cite{bc}).

Suppose that $(X, \Sigma)$ is an $n$-dimensional two strata pseudomanifold with non-empty singular stratum $\Sigma$.
Then, as explained at the beginning of Section 9 in \cite{bc}, the Thom-Mather control data of $\Sigma \subset X$ can be used to construct an open neighborhood $U$ of $\Sigma$ in $X$ and a smooth (locally trivial) fiber bundle $\pi \colon E \rightarrow B$ with the following properties.
The complement $M = X \setminus U$ is a smooth $n$-manifold with boundary $\partial M = E$, and there exists a homeomorphism from the closure $U \cup \partial M$ of $U$ in $X$ to the homotopy pushout $DE$ of $B \stackrel{\pi}{\longleftarrow} E \stackrel{\operatorname{id}_{E}}{\longrightarrow} E$ that restricts to diffeomorphisms $\Sigma \cong B$ and $(U \cup \partial M) \setminus \Sigma \cong DE \setminus B = E \times (0, 1]$ extending $\partial M = E \cong E \times \{0\}$.
We assume that the data $(U, \pi \colon E \rightarrow B)$ have been fixed, and call $U$ a regular neighborhood, and the fiber bundle $\pi \colon E \rightarrow B$ the link bundle of the singular stratum $\Sigma = B$.
The (non-empty) fiber $L$ of $\pi$ is called the link of the singular stratum $\Sigma = B$.

More generally (see Definition 8.3 in \cite{bc}), an $n$-dimensional stratified pseudomanifold of depth $1$ is a tuple $(X, \Sigma_{1}, \dots, \Sigma_{r})$, where the $\Sigma_{i}$ are mutually disjoint subspaces of $X$ such that for every $i = 1, \dots, r$, the pair $(X \setminus \bigcup_{j \neq i} \Sigma_{j}, \Sigma_{i})$ is a two strata pseudomanifold whose singular stratum $\Sigma_{i}$ has regular neighborhood $U^{(i)}$, and link bundle $\pi^{(i)} \colon E^{(i)} \rightarrow B^{(i)}$ having link $L^{(i)}$.
The stratified depth $1$ pseudomanifold $(X, \Sigma_{1}, \dots, \Sigma_{r})$ is called oriented if the top stratum $X \setminus \bigcup_{i}\Sigma_{i}$ is equipped with an orientation.

Consider an $n$-dimensional stratified pseudomanifold $X = (X, \Sigma_{1}, \dots, \Sigma_{r})$ of depth $1$.
Given a perversity $\overline{p}$, we proceed to explain the construction of the $\overline{p}$-intersection space $I^{\overline{p}}X$, which will depend on the choice of equivariant Moore approximations of the link bundles of the singular strata.
We denote by $c_{i}$ the dimension of the link $L^{(i)}$ of the singular stratum $\Sigma_{i}$, and set $k_{i} = c_{i} - \overline{p}(c_{i} + 1)$, which is a positive integer.
For every $i$ we assume that the link $L^{(i)}$ possesses for some choice of structure group $G^{(i)}$ of the link bundle $\pi^{(i)}$ a $G^{(i)}$-equivariant Moore approximation of degree $k_{i}$, say
$$
f_{< k_{i}}^{(i)} \colon L_{< k_{i}}^{(i)} \rightarrow L^{(i)}.
$$
As explained in the beginning of \Cref{truncation cones}, the Moore approximations $f_{< k_{i}}^{(i)}$ induce fiberwise truncations
$$
F_{<k_{i}}^{(i)} \colon \operatorname{ft}_{< k_{i}} E^{(i)} \rightarrow E^{(i)},
$$
where $\operatorname{ft}_{<k_{i}}E^{(i)}$ is the total space of the fiber bundle $\pi_{<k_{i}}^{(i)} \colon \operatorname{ft}_{<k_{i}} E^{(i)} \rightarrow B^{(i)}$ obtained by replacing the fiber $L^{(i)}$ of $\pi^{(i)}$ with the fiber $L_{<k_{i}}^{(i)}$, and $F_{<k_{i}}^{(i)} \colon \operatorname{ft}_{<k_{i}} E^{(i)} \rightarrow E^{(i)}$ is induced by $f_{<k_{i}}^{(i)} \colon L_{<k_{i}}^{(i)} \rightarrow L^{(i)}$ in such a way that $\pi^{(i)} \circ F_{<k_{i}}^{(i)} = \pi_{<k_{i}}^{(i)}$.
Observe that $M = X \setminus \bigcup_{i}U^{(i)}$ is a smooth manifold with boundary $\partial M = \bigsqcup_{i}E^{(i)}$.

\begin{definition}[compare Definition 9.1 in \cite{bc}]\label{definition intersection space}
The \emph{perversity $\overline{p}$ intersection space $I^{\overline{p}}X$} of the depth $1$ pseudomanifold $X = (X, \Sigma_{1}, \dots, \Sigma_{r})$ is defined as the homotopy cofiber of the composition
$$
\bigsqcup_{i} \operatorname{ft}_{< k_{i}} E^{(i)} \stackrel{\bigsqcup_{i} F_{<k_{i}}^{(i)}}{\longrightarrow} \bigsqcup_{i} E^{(i)} = \partial M \stackrel{\operatorname{incl}}{\longrightarrow} M.
$$
In other words, $I^{\overline{p}}X = \operatorname{cone}(\bigsqcup_{i} F_{<k_{i}}^{(i)}) \cup_{\partial M} M$ is the homotopy pushout of
$$
\operatorname{pt} \longleftarrow \bigsqcup_{i} \operatorname{ft}_{< k_{i}} E^{(i)} \stackrel{\tau}{\longrightarrow} M.
$$
\end{definition}

From now on, we are concerned with Witt spaces, which are an important class of stratified pseudomanifolds defined by Siegel \cite{sieg}.

\begin{definition}[see Definition 8.3 in \cite{bc}]\label{definition witt space}
An oriented depth $1$ stratified pseudomanifold $(X, \Sigma_{1}, \dots, \Sigma_{r})$ is called \emph{Witt space} if the following condition is satisfied.
For each $1 \leq i \leq r$ such that the dimension $c_{i}$ of the link $L^{(i)}$ of the singular stratum $\Sigma_{i}$ is even, we have $H_{\frac{c_{i}}{2}}(L^{(i)}) = 0$.
\end{definition}

We specialize to the case of the lower middle perversity $\overline{p} = \overline{m}$ and the upper middle perversity $\overline{q} = \overline{n}$.
Consider a compact depth one Witt space $(X, \Sigma_{1}, \dots, \Sigma_{r})$.
Suppose that $f_{<}^{(i)} \colon L_{<}^{(i)} \rightarrow L^{(i)}$ is a $G^{(i)}$-equivariant Moore approximation of degree $k_{i} = c_{i} - \overline{n}(c_{i} + 1) = \lfloor \frac{1}{2}(c_{i} + 1) \rfloor$.
Then, it follows from the homology vanishing condition of \Cref{definition witt space} that $f_{<}^{(i)}$ is also a $G^{(i)}$-equivariant Moore approximation of $L^{(i)}$ of complementary degree $l_{i} = c_{i} - \overline{m}(c_{i} + 1) = \lceil \frac{1}{2}(c_{i} + 1) \rceil$.
Thus, the resulting intersection spaces $I^{\overline{m}}X$ and $I^{\overline{n}}X$ of \Cref{definition intersection space} can be chosen to be equal, $I^{\overline{m}}X = I^{\overline{n}}X$.

Let us state the main result of this section.

\begin{theorem}\label{main corollary}
Let $(X, \Sigma_{1}, \dots, \Sigma_{r})$ be a compact depth one Witt space of dimension $n \geq 3$.
Assume that for every $i = 1, \dots, r$ the link $L^{(i)}$ of the singular stratum $\Sigma_{i}$ admits an equivariant Moore approximation $f_{<}^{(i)} \colon L_{<}^{(i)} \rightarrow L^{(i)}$ of degree $k_{i} = \lfloor \frac{1}{2}(c_{i} + 1) \rfloor$ (and thus, also one of degree $l_{i} = \lceil \frac{1}{2}(c_{i} + 1) \rceil$).
Let $F_{<}^{(i)} \colon \operatorname{ft}_{<} E^{(i)} \rightarrow E^{(i)}$ denote the fiberwise truncation of the link bundle $\pi^{(i)} \colon E^{(i)} \rightarrow B^{(i)}$.
Then, for the resulting middle perversity intersection space
$$
IX = \operatorname{cone}(\bigsqcup_{i} F_{<}^{(i)}) \cup_{\partial M} M
$$
of \Cref{definition intersection space}, the following statements hold:
\begin{enumerate}[(a)]
\item If the rational Hurewicz homomorphism
$$
\operatorname{Hur}_{n-1 \ast} \colon \pi_{n-1}(\operatorname{cone}(F_{<}^{(i)}), \operatorname{pt}) \rightarrow H_{n-1}(\operatorname{cone}(F_{<}^{(i)}))
$$
is surjective for every $i$, then the middle perversity intersection space $IX$ admits a completion $\widehat{IX} = IX \cup e^{n}$ to a rational Poincar\'{e} duality space by attaching a single $n$-cell.
The rational homotopy type of $\widehat{IX}$ is determined by the intersection space $IX$ whenever a theorem of Stasheff \cite{sta} is applicable.
\item Fix an index $i \in \{1, \dots, r\}$.
If the rational Hurewicz homomorphism in part $(a)$ is surjective, then the local duality obstructions $\mathcal{O}_{\ast}(\pi^{(i)}, k_{i}, l_{i})$ vanish.
The converse implication holds at least when the truncation cone, $\operatorname{cone}(F_{<}^{(i)})$, is simply connected (see \Cref{lemma simply connected cone}), and
\begin{align*}
\operatorname{dim}B^{(i)} < \begin{cases}
(\operatorname{dim} L^{(i)}+1)/2, \quad &\text{if } \operatorname{dim} L^{(i)} \text{ is odd}, \\
(\operatorname{dim} L^{(i)}+4)/2, \quad &\text{if } \operatorname{dim} L^{(i)} \text{ is even}.
\end{cases}
\end{align*}
\item Suppose that the dimension of $X$ is of the form $n = 4d$.
Furthermore, suppose that the rational Hurewicz homomorphism in part $(a)$ is surjective for every $i$, and let $\widehat{IX} = IX \cup e^{n}$ be a completion to a Poincar\'{e} duality space as provided by par $(a)$.
Then, the Witt element $w_{HI} \in W(\mathbb{Q})$ induced by the symmetric intersection form $H_{2d}(\widehat{IX}) \times H_{2d}(\widehat{IX}) \rightarrow \mathbb{Q}$ of the Poincar\'{e} duality space $\widehat{IX}$ equals the Witt element $w_{IH} \in W(\mathbb{Q})$ induced by the Goresky-MacPherson-Siegel intersection form $IH_{2d}(X) \times IH_{2d}(X) \rightarrow \mathbb{Q}$ on middle-perversity intersection homology (see Section I.4.1 in \cite{sieg}).
In particular, the two intersection forms have equal signatures.
\end{enumerate}
\end{theorem}

\begin{proof}
$(a)$.
Fix $i \in \{1, \dots, r\}$.
In the following, we let $f^{(i)} \colon Y^{(i)} \rightarrow Z^{(i)}$ denote the map $F_{<}^{(i)} \colon \operatorname{ft}_{<} E^{(i)} \rightarrow E^{(i)}$, and set $X^{(i)} = \operatorname{cone}(f^{(i)})$.
By the same argument as in the first half of the proof of \Cref{main theorem} (see \Cref{proof of main theorem}), we can show that the map $f^{(i)} \colon Y^{(i)} \rightarrow Z^{(i)}$ and its mapping cone $X^{(i)}$ satisfy all the assumptions of \Cref{proposition poincare duality pairs}.
Then, it follows that the assumptions of \Cref{proposition poincare duality pairs} are also satisfied for the map $f \colon Y \rightarrow Z$ given by $\bigsqcup_{i} f^{(i)} \colon \bigsqcup_{i}Y^{(i)} \rightarrow \bigsqcup_{i}Z^{(i)}$, and its mapping cone $X = \bigvee_{i} X^{(i)}$.
In the following, we will check statement $(ii)$ of \Cref{proposition poincare duality pairs} for the map $f \colon Y \rightarrow Z$ and its mapping cone $X$.
Then, the claim of statement $(a)$ follows from the implication $(ii) \Rightarrow (i)$ of \Cref{proposition poincare duality pairs} and the gluing principle for Poincar\'{e} duality pairs (see Theorem 3.3 in \cite{klim}).

Let $x^{(i)} \in H_{n-1}(X^{(i)})$ denote the image of the orientation class $[Z^{(i)}] \in H_{n-1}(Z^{(i)})$ of the Poincar\'{e} space $Z^{(i)} = E^{(i)}$ under the map $H_{n-1}(Z^{(i)}) \rightarrow H_{n-1}(X^{(i)})$ induced by the inclusion $Z^{(i)} \subset X^{(i)}$.
Then, by the implication $(i) \Rightarrow (ii)$ of \Cref{lemma hurewicz map}, there exists a map $\phi^{(i)} \colon (S^{n-1}, s_{0}) \rightarrow (X^{(i)}, \operatorname{pt})$ such that the element $x^{(i)} \in H_{n-1}(X^{(i)})$ is contained in the image of the connecting homomorphism $\partial_{n} \colon H_{n}(\operatorname{cone}(\phi^{(i)}), X^{(i)}) \rightarrow H_{n-1}(X^{(i)})$.
Hence, by the implication $(iii) \Rightarrow (ii)$ of \Cref{proposition poincare duality pairs} applied to the map $f^{(i)} \colon Y^{(i)} \rightarrow Z^{(i)}$ and its mapping cone $X^{(i)}$, the orientation class $[Z^{(i)}] \in H_{n-1}(Z^{(i)})$ of the Poincar\'{e} space $Z^{(i)} = E^{(i)}$ lies in the image of the connecting homomorphism $\partial_{n} \colon H_{n}(\operatorname{cone}(\phi^{(i)}), Z^{(i)}) \rightarrow H_{n-1}(Z^{(i)})$.
Pick basepoints $y_{i} \in Y^{(i)}$, and let $\overline{Z} \subset X$ denote the mapping cone of the restriction $f| \colon \{y_{1}, \dots, y_{r}\} \rightarrow \bigsqcup_{i} Z^{(i)} = Z$.
For $i = 1, \dots, r$ we choose pairwise disjoint embeddings $\iota_{i} \colon D^{n-1} \rightarrow S^{n-1} \setminus \{s_{0}\}$.
After choosing identifications $S^{n-1}/(S^{n-1} \setminus \iota_{i}(D^{n-1} \setminus \partial D^{n-1})) = S^{n-1}$ of pointed spaces, the quotient map $S^{n-1} \rightarrow S^{n-1}/(S^{n-1} \setminus \iota_{i}(D^{n-1} \setminus \partial D^{n-1}))$ can be considered as a map $q^{(i)} \colon (S^{n-1}, s_{0}) \rightarrow (S^{n-1}, s_{0})$ of degree $1$.
Let
$$
\phi \colon (S^{n-1}, s_{0}) \rightarrow X = \bigvee_{i} X^{(i)}
$$
denote the map uniquely determined by $r^{(i)} \circ \phi = \phi^{(i)} \circ q^{(i)}$ for all $i = 1, \dots, r$, where $r^{(i)} \colon X = \bigvee_{i} X^{(i)} \rightarrow X^{(i)}$ is the collapse map given by $r^{(i)}|_{X^{(i)}} = \operatorname{id}_{X^{(i)}}$ and $r^{(i)}|_{X^{(j)}} = \operatorname{pt} \in X^{(i)}$ for $j \neq i$.
Let $R^{(i)} \colon \operatorname{cone}(\phi^{(i)}) \rightarrow \operatorname{cone}(\phi)$ denote the map induced by the pair $(q^{(i)}, r^{(i)})$.
Note that $R^{(i)}$ restricts to the collapse map $S^{(i)} \colon \overline{Z} \rightarrow Z^{(i)}$ that extends $\operatorname{id}_{Z^{(i)}}$ by the constant map to the basepoint $f(y_{i}) \in Z^{(i)}$.
Writing $\rho_{i} = R^{(i)}_{\ast}$ and $\sigma_{i} = S^{(i)}_{\ast}$ for the induced maps on homology, we have the commutative diagram
\begin{center}
\begin{tikzpicture}

\path (0,0) node(a1) {$H_{n}(\operatorname{cone}(\phi), Z)$}
         (4, 0) node(a2) {$H_{n}(\operatorname{cone}(\phi), \overline{Z})$}
         (9, 0) node(a3) {$\bigoplus_{i} H_{n}(\operatorname{cone}(\phi^{(i)}), Z^{(i)})$}
         (0,-1.5) node(b1) {$H_{n-1}(Z)$}
         (4, -1.5) node(b2) {$H_{n-1}(\overline{Z})$}
         (9, -1.5) node(b3) {$\bigoplus_{i} H_{n-1}(Z^{(i)})$,};

\draw[->] (a1) -- node[above] {$\operatorname{incl}_{\ast}$} node[below] {$\cong$} (a2);
\draw[->] (a2) -- node[above] {$(\rho_{1}, \dots, \rho_{r})$} node[below] {$\cong$} (a3);

\draw[->] (b1) -- node[above] {$\operatorname{incl}_{\ast}$} node[below] {$\cong$} (b2);
\draw[->] (b2) -- node[above] {$(\sigma_{1}, \dots, \sigma_{r})$} node[below] {$\cong$} (b3);

\draw[->] (a1) -- node[left] {$\partial_{n}$} (b1);
\draw[->] (a2) -- node[left] {$\partial_{n}$} (b2);
\draw[->] (a3) -- node[left] {$\oplus_{i} \partial_{n}$} (b3);
\end{tikzpicture}
\end{center}
where all horizontal arrows are isomorphisms by the Eilenberg-Steenrod axioms because $n-1>0$.
Note that the orientation class $[Z] \in H_{n-1}(Z)$ of the Poincar\'{e} space $Z = E$ is mapped to the tuple $([Z^{(1)}], \dots, [Z^{(r)}]) \in \bigoplus_{i} H_{n-1}(Z^{(i)})$ by the lower horizontal arrows.
Hence, it follows that the orientation class $[Z] \in H_{n-1}(Z)$ of the Poincar\'{e} space $Z = E$ lies in the image of the connecting homomorphism $\partial_{n} \colon H_{n}(\operatorname{cone}(\phi), Z) \rightarrow H_{n-1}(Z)$, which is the desired statement $(ii)$ in \Cref{proposition poincare duality pairs}.

$(b)$. In view of \Cref{main theorem}, the claim follows directly from \Cref{proposition rational poincare duality space}.

$(c)$.
In view of Corollary 10.2 in \cite{bc} it suffices to show that in the Witt group of $\mathbb{Q}$, $W(\mathbb{Q})$, the Witt element induced by the intersection form of the Poincar\'{e} space $\widehat{IX}$ equals the Witt element induced by the intersection form of the manifold with boundary $(M, \partial M)$.
The proof is analogous to the proof of Corollary 3.10 in \cite{klim}, which covers the case of isolated singularities.
Generalizing to singular strata of arbitrary dimension, we still apply Novikov additivity of Witt elements under gluing of Poincar\'{e} duality pairs (see Lemma 3.4 in \cite{klim}), where we now have to use the fact that the inclusion $Z \subset X$ induces a surjective homomorphism $H_{2d}(Z) \rightarrow H_{2d}(X)$, which holds by condition (1) of \Cref{proposition poincare duality pairs}.
\end{proof}

\begin{remark}
Passing from Witt elements to signatures in our proof of \Cref{main corollary}(c), we recover the statement of Theorem 11.3 in \cite{bc}, saying that the signature of the intersection form on $H_{2d}(\widehat{IX}) = H_{2d}(IX)$ equals the so-called Novikov signature of the manifold-with-boundary $(M, \partial M)$.
Note that the latter is shown in \cite{bc} for the intersection space of a closed oriented two strata Witt space under the assumption that the local duality obstructions of the link bundle vanish.
The proof given in \cite{bc} is involved, and requires to construct an intersection form of the intersection space $IX$ that is symmetric.
As in our proof of \Cref{main corollary}(c), one also exploits surjectivity of the map $H_{2d}(Z) \rightarrow H_{2d}(X)$, namely in diagram (11.2) in \cite{bc}, where the map takes the form $C_{\geq k \ast} \colon H_{m}(\partial M) \rightarrow H_{m}(Q_{\geq k}E)$.
\end{remark}

\subsection{An example}\label{An example}
To illustrate a non-trivial case in which \Cref{main corollary} applies, we discuss a class of examples of Witt spaces having as singular strata a finite number of circles with twisted link bundles.
As input data for our construction, we employ commutative diagrams of the form
\begin{center}
\begin{tikzpicture}
\path (0,0) node(a) {$L_{<k}$}
         (3, 0) node(b) {$L$}
         (0, -1.5) node(c) {$L_{<k}$}
         (3, -1.5) node(d) {$L$,};

\draw[->] (a) -- node[left] {$\lambda_{<k}$} (c);
\draw[->] (a) -- node[above] {$f_{<}$} (b);
\draw[->] (b) -- node[right] {$\lambda$} (d);
\draw[->] (c) -- node[above] {$f_{<}$} (d);
\end{tikzpicture}
\end{center}
where
\begin{itemize}
\item $L$ is a closed oriented smooth manifold of dimension $c > 0$,
\item $f_{<} \colon L_{<} \rightarrow L$ is a Moore approximation of $L$ of degree $\lfloor\frac{1}{2}(c + 1) \rfloor$,
\item $\lambda \colon L \rightarrow L$ is an orientation preserving diffeomorphism,
\item $\lambda_{<} \colon L_{<} \rightarrow L_{<}$ is a homeomorphism,
\end{itemize}
such that
\begin{itemize}
\item $H_{\frac{c}{2}}(L) = 0$ when $c$ is even,
\item $f_{<}$ induces a surjection $\pi_{1}(L_{<}, x_{0}) \rightarrow \pi_{1}(L, f_{<}(x_{0}))$ for every $x \in L_{<}$, and
\item $\lambda^{N} = \operatorname{id}_{L}$ and $\lambda_{L}^{N} = \operatorname{id}_{L}$ for some integer $N > 0$.
\end{itemize}

Given a commutative diagram as above, we consider the mapping torus
$$
E = (L \times [0, 1])/(x, 0) \sim (\lambda(x), 1)
$$
as the total space of a smooth fiber bundle $\pi \colon E \rightarrow S^{1}$ with fiber $L$ and structure group $G = \mathbb{Z}/N\mathbb{Z}$ acting on $L$ via $i + N\mathbb{Z} \mapsto \lambda^{i}$.
The bundle $\pi$ is \emph{flat} in the sense that we can choose $G$-valued transition functions that are locally constant.
Recall that the integers $k = \lfloor\frac{1}{2}(c + 1) \rfloor$ and $l = \lceil\frac{1}{2}(c + 1) \rceil$ are associated to the upper-middle perversity $\overline{n}$ and the lower-middle perversity $\overline{m}$, which are complementary.
By the properties of the above diagram, the map $f_{<} \colon L_{<} \rightarrow L$ is a $G$-equivariant Moore approximation both of degree $k$ and of degree $l$.
It can be shown that the local duality obstructions $\mathcal{O}_{\ast}(\pi, k, l)$ vanish.
(In fact, we can follow the proof of Theorem 7.1 in \cite{bc}, which applies literally when we replace the universal cover $\widetilde{B} \rightarrow B$ by the $N$-sheeted cover $S^{1} \rightarrow S^{1}$ having  the finite group $\pi_{1} = \mathbb{Z}/N$ as its group of deck transformations.)
Let $F_{<} \colon \operatorname{ft}_{<} E \rightarrow E$ denote the fiberwise truncation induced by the $G$-equivariant Moore approximation $f_{<} \colon L_{<} \rightarrow L$.
The truncation cone of $f_{<}$, $\operatorname{cone}(F_{<})$, is simply connected by \Cref{lemma simply connected cone}.
Hence, \Cref{main corollary}(b) implies that the rational Hurewicz homomorphism
$$
\operatorname{Hur}_{n-1 \ast} \colon \pi_{n-1}(\operatorname{cone}(F_{<}), \operatorname{pt}) \rightarrow H_{n-1}(\operatorname{cone}(F_{<}))
$$
is surjective.

We choose a suitable finite disjoint union $\bigsqcup_{i} E^{(i)}$ of mapping tori $E^{(i)}$ coming from the previous construction that can be realized as the boundary of a compact oriented smooth manifold $M$ of dimension $n = c+2$.
For instance, when $c = 6$, then we can use a single mapping torus because $\Omega_{7}^{SO} = 0$.
Denoting the homotopy pushout of $S^{1} \stackrel{\pi}{\longleftarrow} E^{(i)} \stackrel{\operatorname{id}_{E^{(i)}}}{\longrightarrow} E^{(i)}$ by $DE^{(i)}$, we thus obtain a compact depth one Witt space $X = M \cup_{\partial M} \bigsqcup_{i} DE^{(i)}$ of dimension $n \geq 3$.
Its middle perversity intersection space is given by
$$
IX = \operatorname{cone}(\bigsqcup_{i} F_{<}^{(i)}) \cup_{\partial M} M,
$$
and according to \Cref{main corollary}(a), there is a completion $\widehat{IX} = IX \cup e^{n}$ to a rational Poincar\'{e} duality space by attaching a single $n$-cell.
The rational homotopy type of $\widehat{IX}$ is determined by the intersection space $IX$ whenever a theorem of Stasheff \cite{sta} is applicable.
By \Cref{main corollary}(c), the signature of the Poincar\'{e} duality space $\widehat{IX}$ agrees with the signature of the Goresky-MacPherson-Siegel intersection form on middle-perversity intersection homology of $X$.

In conclusion, let us discuss an example for a commutative diagram as above.

\begin{example}
Consider a closed smooth manifold $K$ of dimension $c \geq 3$ such that $H_{\frac{c}{2}}(K) = 0$ when $c$ is even.
We fix a $c$-dimensional triangulation $K^{0} \subset \dots \subset K^{c}$ of $K$.
Set $k = \lfloor\frac{1}{2}(c + 1) \rfloor$.
The $(k-1)$-skeleton $K^{k-1}$ can be extended to a $k$-dimensional $CW$ complex $K_{<}$ in such a way that there exists a Moore approximation $g_{<} \colon K_{<} \rightarrow K^{k}$ of degree $k$ which restricts to the identity map on $K^{k-1}$.
(When $K$ is simply connected and $k \geq 3$, this can be achieved by using Proposition 1.6 of \cite{ban}.
More generally, note that Proposition 1.3 in \cite{wra} can be applied whenever $k \geq 2$, and without assuming $K$ to be simply connected.)
The connected sum $L = K \sharp K$ is a closed smooth $c$-manifold which admits an orientation preserving diffeomorphism $\lambda \colon L \rightarrow L$ interchanging the two summands in such a way that $\lambda^{2} = \operatorname{id}_{L}$.
We may equip $L$ with a $c$-dimensional CW structure $L^{0} \subset \dots \subset L^{c}$ in such a way that $L^{c-1} = K^{c-1} \vee_{x_{0}} K^{c-1}$ for some basepoint $x_{0} \in K^{0}$, and such that $\lambda$ restricts to the homeomorphism $L^{c-1} \rightarrow L^{c-1}$ that interchanges the two copies of $L^{c-1}$.
(To achieve this, we fix a $c$-simplex $\Delta_{c}$ of $K$, choose an embedded closed unit ball $B^{c}$ in the interior of $\Delta_{c}$, and delete the interior $U^{c}$ of $B^{c}$.
Then, we form the connected sum $L = K \sharp K$ by gluing two copies of $K \setminus U^{c}$ via the identity map on $\partial B^{c}$.
In order to find the desired CW structure on $L$, we modify $B^{c}$ by moving one point of $\partial B^{c}$ to a $0$-simplex $\{x_{0}\} \subset \partial \Delta_{c}$.
Then, we see that $L^{c-1} = K^{c-1} \vee_{x_{0}} K^{c-1}$, and the $c$-cells of $L$ are given by the $c$-simplices of the two copies of $K$ that are different from $\Delta_{c}$, plus one new $c$-cell whose attaching map arises from gluing the two copies of $\Delta_{c}$ along the modified $\partial B^{c}$.)
We define a degree $k$ Moore approximation $f_{<} \colon L_{<} \rightarrow L$ by taking the composition
$$
K_{<} \vee_{x_{0}} K_{<} \stackrel{g_{<} \vee g_{<}}{\longrightarrow} K^{k} \vee_{x_{0}} K^{k} = L^{k} \hookrightarrow L = K \sharp K.
$$
Finally, we define $\lambda_{<} \colon L_{<} \rightarrow L_{<}$ to be the homeomorphism that interchanges the two copies of $K_{<}$ in the bouquet $L_{<} = K_{<} \vee_{x_{0}} K_{<}$.
Then, the desired diagram commutes by construction.
\end{example}

\end{document}